\documentclass[12pt,a4paper]{article}

\usepackage{mathptmx}       
\usepackage{helvet}         
\usepackage{courier}        
\usepackage{type1cm}        
\usepackage{makeidx}         
\usepackage{graphicx}        
\usepackage{multicol}        
\usepackage[bottom]{footmisc}

\usepackage{amssymb}
\usepackage{amsfonts}
\usepackage{amsmath}
\usepackage{amsthm}
\usepackage{epsfig}
\usepackage{graphics}
\usepackage{dsfont}

\usepackage{amsfonts}
\usepackage{amsmath}
\usepackage{epsfig}
\usepackage{dsfont}
\allowdisplaybreaks[3]

\newtheorem{theorem}{Theorem}
\newtheorem{lemma}{Lemma}
\newtheorem{remark}{Remark}
\newtheorem{example}{Example}
\newtheorem{proposition}{Proposition}

\newtheorem{corollary}{Corollary}

\makeindex             

\newcommand{\NN}{\mathbb{N}}

\newcommand{\RR}{\mathbb{R}}
\newcommand{\CC}{\mathbb{C}}

\newcommand{\LL}{\mathbb{L}}

\newcommand{\cB}{\mathcal{B}}

\newcommand{\cP}{\mathcal{P}}

\newcommand{\cN}{\mathcal{N}}
\newcommand{\cE}{\mathcal{E}}

\newcommand{\cL}{\mathcal{L}}

\newcommand{\var}{\mathrm{Var}\,}
\newcommand{\supp}{\mathrm{supp}\,}

\newcommand{\bN}{\mathbb{N}}
\newcommand{\bR}{\mathbb{R}}
\newcommand{\bZ}{\mathbb{Z}}
\newcommand{\bC}{\mathbb{C}}
\newcommand{\law}{\mathcal{L}}

\begin{document}

\title{On the range of exponential functionals of L\'evy processes}
\author{Anita Behme\thanks{Technische Universit\"at Dresden,
Institut f\"ur Mathematische Stochastik, D-01062 Dresden, Germany and  Technische Universit\"at M\"unchen, Zentrum Mathematik,
Boltzmannstra\ss e 3, D-85748 Garching bei M\"unchen,
Germany; email: a.behme@tum.de, tel.: +49/89/28917424,
fax:+49/89/28917435}$\,$  \and Alexander Lindner\thanks{Technische
Universit\"at Braunschweig, Institut f\"ur Mathematische Stochastik,
D-38106 Braunschweig, Germany, and 
Ulm University, Institute of Mathematical Finance, D-89081 Ulm, Germany; 
email:
alexander.lindner@uni-ulm.de} \and
Makoto Maejima\thanks{Keio University, Department of Mathematics,
Hiyoshi, Yokohama 223-8522, Japan, email: maejima@math.keio.ac.jp}}
%
%
\maketitle

\abstract{We characterize the support of the law of the exponential functional $\int_0^\infty e^{-\xi_{s-}} \, d\eta_s$ of two one-dimensional independent L\'evy processes $\xi$ and $\eta$. Further, we  study the range of the mapping $\Phi_\xi$ for a fixed L\'evy process $\xi$, which maps the law of $\eta_1$ to the law of the corresponding exponential functional
$\int_0^\infty e^{-\xi_{s-}} \, d\eta_s$. It is shown that the range of this mapping is closed under weak convergence and in the special case of positive distributions several characterizations of laws in the range are given.}

\section{Introduction}

Given a bivariate L\'evy process $(\xi,\eta)^T=((\xi_t,\eta_t)^T)_{t\geq 0}$, its
{\it exponential functional} is defined as
\begin{equation} \label{eq-expfunc}
V:=\int_0^\infty e^{-\xi_{s-}} d\eta_s,
\end{equation}
provided that the integral converges almost surely.
Exponential functionals of L\'evy processes appear as stationary
distributions of generalized Ornstein-Uhlenbeck (GOU) processes.
In particular, if $\xi$ and $\eta$ are independent and $\xi_t$ tends to $+\infty$ as $t\to\infty$
almost surely,
then the law of $V$ defined in \eqref{eq-expfunc} is the stationary distribution of the GOU process
\begin{equation} \label{GOUdef}
V_t=e^{-\xi_t} \left( \int_0^t e^{\xi_{s-}}d\eta_s +V_0  \right),\quad t\geq 0,
\end{equation}
where $V_0$ is a starting random variable, independent of $(\xi,\eta)^T$, on
the same probability space (cf. \cite[Theorem 2.1]{lindnermaller05}). Hence, when $V_0$ is chosen to have the same distribution as $V$,
then the process $(V_t)_{t\geq 0}$ is strictly stationary.

Unless $\xi_t = at$ with $a>0$, the distribution of $V$ is known only in a few special cases. See e.g. Bertoin and Yor \cite{BertoinYor} for a survey on exponential functionals of the form $V=\int_0^\infty e^{-\xi_{s-}} \, ds$, or Gjessing and Paulsen \cite{GjessingPaulsen1997}, who determine the distribution of $\int_0^\infty e^{-\xi_{s-}} \, d\eta_s$ for some cases. A thorough study of distributions of the form $\int_0^\infty e^{-\xi_{s-}} \, d\eta_s$ when $\eta$ is a Brownian motion is carried out in Kuznetsov et al. \cite{savovetal}.
We state the following example due to Dufresne (e.g. \cite[Equation (16)]{BertoinYor}) of an exponential functional whose distribution has been determined and to which we will refer later. Here and in the following we write ``$\stackrel{d}{=}$'' to denote equality in distribution of random variables.

\begin{example} \label{rem-Dufresne}
For $(\xi_t, \eta_t)=(\sigma B_t+a t, t)$ with $\sigma >0$, $a>0$ and a standard Brownian motion $(B_t)_{t\geq 0}$ it holds
$$
V=\int_0^\infty e^{-(\sigma B_t+a t)} dt\overset{d}= \frac{2}{\sigma^2 \Gamma_{\frac{2a}{\sigma^2}}},
$$
where $\Gamma_r$ denotes a standard Gamma random variable with shape parameter $r$, i.e. with density
$$
P(\Gamma_r \in dx)=\frac{x^{r-1}}{\Gamma(r)}e^{-x} \mathds{1}_{(0,\infty)}(x) dx.
$$
\end{example}

Denote by $\cL(X)$ the law of a random variable $X$ and let $\xi=(\xi_t)_{t\geq 0}$ be a
one-dimensional L\'evy process drifting to $+\infty$. In this paper we will consider the mapping
\begin{align*}
\Phi_\xi : D_\xi  & \to \cP(\bR) := \mbox{the set of probability distributions on $\bR$} ,\\
\cL(\eta_1) & \mapsto \cL \left( \int_0^\infty e^{-\xi_{s-}} \,
d\eta_s \right),
\end{align*}
defined on
\begin{align*}
D_\xi := \{ \cL (\eta_1) : \eta=(\eta_t)_{t\ge 0} \, & \mbox{ one-dimensional L\'evy process independent of }\xi\\
& \hskip 15mm \mbox{ such that }\int_0^\infty
e^{-\xi_{s-}} \, d\eta_s\mbox{ converges a.s.}\}.
\end{align*}
An explicit description of $D_\xi$ in terms of the characteristic triplets (cf. \eqref{levykhintchine}) of $\xi$ and $\eta$ follows from
Theorem 2 in Erickson and Maller \cite{ericksonmaller05}. Denote the range of $\Phi_\xi$ by
$$R_\xi := \Phi_\xi(D_\xi).$$
Although the domain $\Phi_\xi$ can be completely characterized by \cite{ericksonmaller05}, much less is known about the range $R_\xi$
and properties of the mapping $\Phi_\xi$.
In the case that $\xi_t=at, a>0$ is deterministic, it is well known that $D_\xi=\mbox{ID}_{\log}(\RR)$,
the set of real-valued infinitely divisible distributions with finite $\mbox{log}^+$-moment, and that
$\Phi_\xi$ is an algebraic isomorphism between $\mbox{ID}_{\log}(\RR)$ and $R_{\xi}= L(\RR)$,
the set of real-valued selfdecomposable distributions \cite[Proposition 3.6.10]{jurekmason}.

For general $\xi$, the mapping $\Phi_\xi$ has already been studied in \cite{BehmeLindner13},
where it has been shown that $\Phi_\xi$ is injective in many cases,
while injectivity cannot be obtained if $\xi$ and $\eta$ are
allowed to exhibit a dependence structure. Further in \cite{BehmeLindner13} conditions
for continuity (in a weak sense) of $\Phi_\xi$ are given.
These results were then used to obtain some information on the range $R_\xi$. 
In particular it has been shown that centered Gaussian distributions can only be obtained in the
setting of (classical) OU processes, namely, for $\xi$ being deterministic and $\eta$
being a Brownian motion.

In this paper we take up the subject of studying properties of the mapping $\Phi_\xi$ and of distributions in $R_\xi$,
and start in Section \ref{S2} with a classification of possible supports of the laws in $R_\xi$. Section \ref{S3} is devoted to show closedness of the range $R_\xi$ under weak convergence. It also follows that the inverse mapping $\Phi_\xi^{-1}$ is continuous if it is well-defined, i.e. if $\Phi_\xi$ is injective. In Sections \ref{S4} and \ref{S5} we specialize on positive distributions in $R_\xi$. Section \ref{S4} gives a general criterion for positive distributions to belong to  $R_\xi$. In Section \ref{S5} we use this criterion to obtain further results in the case that $\xi$ is a Brownian motion with drift. We derive a differential equation for the Laplace exponent of a positive distribution in $R_\xi$ and from this we gain concrete conditions in terms of L\'evy measure and drift for some distributions to be in $R_\xi$. We end up studying the special case of positive stable distributions in $R_\xi$.

For an $\RR^d$-valued L\'evy process $X=(X_t)_{t\geq 0}$, the {\it characteristic exponent}
is given by its L\'evy-Khintchine formula (e.g. \cite[Theorem 8.1]{sato})
\begin{align} \label{levykhintchine}
\log \phi_X(u) &:= \log E\left[e^{i\langle u, X_1 \rangle } \right]\\
&= i \langle \gamma_X, u\rangle - \frac{1}{2} \langle u, A_X u  \rangle + \int_{\RR^d}
(e^{i \langle u, x \rangle} -1 -i \langle u, x \rangle \mathds{1}_{|x|\leq 1}) \nu_X(dx), \quad u\in \bR, \nonumber
\end{align}
where $(\gamma_X, A_X, \nu_X)$ is the {\it characteristic triplet} of the L\'evy process $X$.
In case that $X$ is real-valued we will usually replace $A_X$ by $\sigma^2_X$.
In the special case of subordinators in $\bR$, i.e. nondecreasing L\'evy processes,
we will also use
the Laplace transform $$\LL_{X} (u) := E [e^{-u X_1}]= e^{\psi_X(u)}, \quad u\geq 0,$$ of $X$ and call
$\psi_X(u)$ the Laplace exponent of the L\'evy process $X$.
We refer to \cite{sato} for  further information on L\'evy processes.
In the following, when the symbol $X$ is regarded as a real-valued random variable, we
also use the notation $\phi_X(u)$ and $\LL_{X} (u)$ for its characteristic function and
Laplace transform, respectively.
The Fourier transform of a finite measure $\mu$ on
$\bR$ is written as $\widehat{\mu}(u) = \int_{\bR} e^{iux} \,\mu(dx)$.
We write ``$\stackrel{d}{\to}$'' to denote convergence in distribution of random variables, and ``$\stackrel{w}{\to}$'' to denote weak convergence of probability measures.
We use the abbreviation ``i.i.d.'' for ``independent and identically distributed''. The set of all twice continuously differentiable functions $f:\bR\to \bR$ which are bounded will be denoted by $C_b^2(\bR)$, and the subset of all $f:\bR\to \bR$ which have additionally compact support by $C_c^2(\bR)$.

\section{On the support of the exponential functional}\label{S2}

In this section we shall  give the support of the distribution of the exponential functional
$V = \int_0^\infty e^{-\xi_{s-}} \, d\eta_s$ when $\xi$ and $\eta$ are independent L\'evy processes.
In particular it turns out that the support will always be a closed interval.
A similar result does not hold for solutions of arbitrary random recurrence equations,
or for exponential functionals of L\'evy processes with dependent $\xi$ and $\eta$, as we shall show in Remark~\ref{rem-support-counter}.

For $\xi$ being spectrally negative, it is well known (e.g. \cite{bertoinlindnermaller08}) that $V$ has a selfdecomposable and hence infinitely divisible distribution. In \cite[Theorem 24.10]{sato} a characterization of the support of infinitely divisible distributions is given in terms of the L\'evy triplet. In particular the support of a selfdecomposable distribution on $\RR$ is either a single point, $\RR$ itself or a one-sided unbounded interval. Unfortunately the characteristic triplet of $V$ is not known in general and also, for not spectrally negative $\xi$ this result can not be applied.

Before we characterize the support of the law of $V=\int_0^\infty e^{-\xi_{s-}} \, d\eta_s$ when $\xi$ and $\eta$
are general independent L\'evy processes, we treat the special case
when $\eta_t=t$ in the following lemma.
Much attention has been paid to this case, and in particular, it has been shown that
the stationary solution has a density under various conditions,
see e.g. Pardo et al. \cite{PardoRiveroSchaik} or Carmona at al. \cite{carmonapetityor}.
Haas and Rivero \cite[Theorem 1.4, Lemma 2.1]{HaasRivero} gave a characterization when this
law is bounded  and obtained that this is the case if and only if $\xi$ is a subordinator
with strictly positive drift, and derived the support then.
So parts of the following lemma follow already from results in \cite{HaasRivero},
nevertheless we have decided to give a detailed proof.

\begin{lemma} \label{lem-support}
Let $\xi$ be a L\'evy process drifting to $+\infty$ and set $V = \int_0^\infty e^{-\xi_{s}} ds$.
Then
$$
\supp \law(V) = \begin{cases} \left\{ \frac{1}{b} \right\},
& \mbox{if $\xi_t = bt$ with $b>0$},\\
\left[0,\frac{1}{b} \right],
& \mbox{if $\xi$ is a non-deterministic subordinator with drift $b>0$},\\
\left[\frac{1}{b}, \infty\right) ,
& \mbox{if $\xi$ is non-deterministic and of finite variation},\\
& \mbox{with drift $b>0$ and $\nu_\xi((0,\infty)) = 0$},\\
[0,\infty), & \mbox{otherwise.}
\end{cases}
$$
\end{lemma}

\begin{proof}
The claim is clear if $\xi$ is deterministic, while it follows from Remark~\ref{rem-Dufresne}
if $\xi$ is a Brownian motion with drift, so suppose that $\nu_\xi \not\equiv 0$.
Suppose first that $\nu_\xi((0,\infty)) > 0$, and let $x_0 \in \supp \law (V) \cap (0,\infty)$.
Let $c\in \supp \nu_\xi \cap (0,\infty)$ and $y_0 \in (e^{-c}x_0, x_0)$.
We shall show that also $y_0 \in \supp \law(V)$, so that by induction $\supp \law(V)$
must be an interval with lower endpoint 0 if $\nu_\xi((0,\infty)) > 0$.
To see this, define $z_0 \in (0,y_0)$
so that
$$
z_0 + e^{-c} (x_0 - z_0) = y_0.
$$
Let $\varepsilon \in (0, \frac{x_0-z_0}{2})$ and define
$$
A = A_\varepsilon := \left\{ \omega \in \Omega : \int_0^\infty e^{-\xi_s(\omega)} \, ds
\in (x_0-\varepsilon, x_0 + \varepsilon)\right\}.
$$
Then $P(A) > 0$ since $x_0 \in \supp \law(V)$.
Define the stopping time $T_1\in [0,\infty]$ by
$$
T_1(\omega) := \inf \left\{ t\geq 0: \int_0^{t} e^{-\xi_s(\omega)} \, ds = z_0\right\} .
$$
Since $t\mapsto \int_0^t e^{-\xi_s(\omega)} \, ds$ is continuous,
$T_1$ is finite on $A$.
Let
$\delta_1 \in (0, \frac{x_0-z_0}{2})$ and  $\delta_2 \in (0,c)$.
Then $\nu_\eta((c-\delta_2,c+\delta_2)) > 0$, and since $P(A) > 0$,
there are a (sufficiently large) constant $K=K(\varepsilon, \delta_1,\delta_2)>0$ and a
(sufficiently small) constant $\delta = \delta(\varepsilon, \delta_1,\delta_2) > 0$
such that $\delta<1$ and
\begin{align*}
B := B_{\varepsilon,\delta_1,\delta_2, \delta, K}   :=   A \cap
  \Big \{T_1  &\leq K, \int_{T_1}^{T_1+\delta} e^{-\xi_s} \, ds \leq \delta_1, \\
&    \Delta \xi_s \not\in (c-\delta_2,c+\delta_2), \; \forall\, s\in (T_1,T_1+\delta] \Big\}
\end{align*}
has a positive probability. Now define the set $C=C_{\varepsilon, \delta_1,\delta_2,\delta,K}$
to be the set of all $\omega \in \Omega$, for which there exists an $\omega'\in B$,
some time $t(\omega') \in (T_1\wedge K , (T_1\wedge K) + \delta]$ and some
$\alpha(\omega') \in (c-\delta_2, c+\delta_2)$ such that
$$
(\xi_t(\omega))_{t\geq 0} = (\xi_t(\omega') + \alpha(\omega') \mathds{1}_{[t(\omega'), \infty)})_{t\geq 0},
$$
namely, the set of $\omega$ whose paths behave exactly like a sample path from the set $B$,
but with the exception that additionally exactly one jump of size in $(c-\delta_2,c+\delta_2)$
occurs in the interval $(T_1\wedge K, (T_1\wedge K) +\delta]$.
Since $T_1\wedge K$ is a finite stopping time, it follows from the strong Markov property of $\xi$
and from $P(B) > 0$ that also $P(C) > 0$.
But for $\omega\in C$, with $\omega'\in B$ and $\alpha=\alpha(\omega') \in (c-\delta_2,c+\delta_2)$ as in
the definition of $C$, we obtain
\begin{align*}
\lefteqn{\int_0^\infty e^{-\xi_s(\omega)} \, ds }\\ &=  \int_0^{T_1(\omega')} e^{-\xi_s (\omega')} \, ds +
\int_{T_1(\omega')}^{T_1(\omega')+\delta} e^{-\xi_s(\omega)}\, ds + e^{-\alpha}
\int_{T_1(\omega')+\delta}^\infty e^{-\xi_s(\omega')} \, ds \\
& \in \Big[ z_0  + \int_{T_1(\omega')}^{T_1(\omega')+\delta} e^{-\xi_s(\omega)}\, ds +
e^{-\alpha} \Big(x_0 - \varepsilon- z_0-
\int_{T_1(\omega')}^{T_1(\omega')+\delta} e^{-\xi_s(\omega')} \, ds \Big),\\
&  \hskip 10mm z_0  + \int_{T_1(\omega')}^{T_1(\omega')+\delta} e^{-\xi_s(\omega)}\, ds +
e^{-\alpha} \Big(x_0 +\varepsilon - z_0-
\int_{T_1(\omega')}^{T_1(\omega')+\delta} e^{-\xi_s(\omega')} \, ds \Big) \Big] \\ &\\ & \\
&\subset  \Big[ z_0  -\delta_1 + e^{-c} (x_0-z_0-\varepsilon) + (e^{-c-\delta_2} - e^{-c})
(x_0-z_0-\varepsilon) - e^{-c+\delta_2} \delta_1, \\ 
&  \hskip 10mm  z_0 +\delta_1 + e^{-c} (x_0-z_0+\varepsilon) + (e^{-c+\delta_2} - e^{-c})
(x_0-z_0+\varepsilon) + e^{-c+
\delta_2} \delta_1\Big].
\end{align*}
Since $y_0 = z_0 + e^{-c} (x_0 - z_0)$, we see that $y_0 \in \supp \law(V)$ by choosing
$\varepsilon, \delta_1$ and $\delta_2$ sufficiently small.
So we have shown that $\supp \law(V)$ is an interval with 0 as its lower endpoint if
$\nu_\xi((0,\infty)) > 0$.\\
By a similar reasoning, one can show that $\supp \law(V)$ is an interval with $+\infty$ as
its upper endpoint if $\nu_\xi((-\infty,0)) > 0$.

It follows that $\supp \law (V) = [0,\infty)$ if $\nu_\xi((0,\infty)) > 0$ and
$\nu_\xi((-\infty,0)) > 0$.
Now suppose that $\xi$ is of infinite variation with
$\nu_\xi ((0,\infty)) > 0$ (but $\nu_\xi((-\infty,0)) = 0$), or
$\nu_\xi((-\infty,0)) > 0$ (but $\nu_\xi((0,\infty))  = 0$).
Then there is  $\alpha > 0$ such that for each $t_1,t_0 > 0$ with $t_1> t_0$ and $K> 0$ the event
$$
\{\xi_s \geq -2, \; \forall s\in [0,t_0], \quad \xi_s \geq K, \; \forall s\in [t_0,t_1] ,\quad
\xi_s \geq \alpha s, \; \forall s \geq t_1\}
$$
has a positive probability, since $\lim_{t\to\infty} t^{-1} \xi_t$ exists almost surely in $(0,\infty]$
by \cite[Theorems 4.3 and 4.4]{doneymaller2002}) and since $\supp \law(\xi_t) = \bR$
for all $t>0$ (cf. \cite[Theorem 24.10]{sato}).
Choosing $t_0$ small enough and $t_1,K$ big enough, it follows that $0 \in \supp \law(V)$ since $\supp \law(V)$ is closed. On the other hand, since also the event
$$
\{\xi_s \leq 2, \; \forall s \in [0,t_2]\}
$$
has positive probability for each $t_2>0$ as a consequence of the infinite variation of $\xi$,
it follows that $\supp \law(V)$ is unbounded, hence showing that $\supp \law(V) = [0,\infty)$
if $\xi$ is of infinite variation.

Now assume that $\xi$ is of finite variation with drift $b\in \bR$, $\nu_\xi((0,\infty)) > 0$
and $\nu_\xi((-\infty,0)) = 0$. We already know that $0 \in \supp \law(V)$.
If $b\leq 0$, then the event $\{\xi_s \leq 2, \; \forall s\in [0,t_2]\}$ has
a positive probability for each $t_2>0$, and hence $\supp \law(V)$ is unbounded.
If $b>0$, then for each $\varepsilon >0$ and $t_2>0$, the event
$\{\xi_s \leq (b+\varepsilon)s, \; \forall s\in [0,t_2]\}$ has a positive probability by
Shtatland's result (cf. \cite[Theorem 43.20]{sato}), so that $\sup \supp \law(V)
\geq \int_0^{t_2} e^{-(b+\varepsilon)s} \, ds$ for each $t_2>0$ and $\varepsilon >0$,
and hence  $\sup \supp \law(V) \geq 1/b$.
On the other hand, in that case $V = \int_0^\infty e^{-\xi_s} \, ds \leq
\int_0^\infty e^{-bs} \, ds = 1/b$, so that $\supp \law(V) = [0,1/b]$.

Finally, assume that $\xi$ is of finite variation with drift $b>0$, $\nu_\xi((0,\infty)) = 0$
and\linebreak $\nu_\xi((-\infty,0)) > 0$.
Then $\supp \law(V)$ is unbounded and by arguments similar to above, using that
$\lim_{t\to\infty} t^{-1} \xi_t = E [\xi_1] \in (0,b)$, we see that $\inf \supp \law(V) = 1/b$,
so that $\supp \law(V) = [1/b,\infty)$.
This finishes the proof.\qed
\end{proof}

Now we can characterize the support of $\law\left(\int_0^\infty e^{-\xi_{s-}} \, d\eta_s\right)$
when $\xi$ and $\eta$ are independent L\'evy processes. Observe that Theorem \ref{thm-support} below together with Lemma~\ref{lem-support}  provides
a complete characterization of all possible cases.

\begin{theorem} \label{thm-support}
Let $\xi$ and $\eta$ be two independent L\'evy processes such that
$V := \int_0^\infty e^{-\xi_{s-}} \, d\eta_s$ converges almost surely.
\begin{enumerate}
\item Suppose $\eta$ is of infinite variation, or that $\nu_\eta ((0,\infty))>0$ and
$\nu_\eta ((-\infty,0))>0$. Then $\supp \law(V)  = \bR$.
\item Suppose $\eta$ is of finite variation with drift $a$, $\nu_\eta((0,\infty)) > 0$ and $\nu_\eta((-\infty,0)) = 0$.
Then for $a\geq 0$
$$
\supp \law(V) =
\begin{cases}
\left[\frac{a}{b}, \infty\right),  &\mbox{if $\xi$ is of finite variation with drift
$b>0$} \\ & \mbox{and $\nu_\xi((0,\infty)) = 0$},\\
 [0,\infty), & \mbox{otherwise},
\end{cases}
$$
and for $a<0$
$$
\supp \law(V) =
\begin{cases}
\left[\frac{a}{b}, \infty\right),  &\mbox{if $\xi$ is a subordinator with drift $b>0$},\\
\bR, & \mbox{otherwise.} \\
\end{cases}
$$
\item Suppose $\eta$ is of finite variation with drift $a$,
$\nu_\eta((0,\infty)) = 0$ and $\nu_\eta((-\infty,0)) > 0$.
Then for $a> 0$
$$
\supp \law(V) =
\begin{cases}
\left( - \infty, \frac{a}{b}\right],  &\mbox{if $\xi$ is a subordinator with drift $b>0$},\\
\bR, & \mbox{otherwise,} \\
\end{cases}
$$
and for $a\leq 0$
$$
\supp \law(V) =
\begin{cases}
\left(-\infty, \frac{a}{b}\right],  &\mbox{if $\xi$ is of finite variation with drift
$b>0$} \\ & \mbox{and $\nu_\xi((0,\infty)) = 0$},\\
 (-\infty,0], & \mbox{otherwise.}
\end{cases}
$$
\end{enumerate}
\end{theorem}

\begin{proof}
Denote by $D([0,\infty),\bR)$ the set of all real valued c\`adl\`ag functions on $[0,\infty)$.
Since $\xi$ and $\eta$ are independent, we can condition on $\xi=f$ with
$f\in D([0,\infty),\bR)$ and it follows that, for $P_\xi$-almost every $f\in D([0,\infty),\bR)$,
$$
V_f := \int_0^\infty e^{-f(s-)} \, d\eta_s = \lim_{T\to\infty} \int_0^T e^{-f(s-)} \, d\eta_s
$$
converges almost surely.
Hence we can apply the results in \cite{sato2007} for such $f$, and obtain that
$V_f$ is infinitely divisible with Gaussian variance
$$
A_f = A_\eta \int_0^\infty e^{-2f(s)} \, ds
$$
and L\'evy measure $\nu_f$, given by
$$
\nu_f(B) = \int_0^\infty ds\int_{\bR} \mathds{1}_B (e^{-f(s)} x) \,
\nu_\eta(dx) \quad \mbox{for $B\in \cB(\bR^d)$ with $0\not\in B$}
$$
(cf. \cite[Theorem 3.10]{sato2007}).
In particular, $A_f > 0$ if and only if $A_\eta>0$,
$\nu_f((0,\infty)) > 0$ if and only if $\nu_\eta((0,\infty)) > 0$,
and $\nu_f((-\infty,0)) > 0 $ if and only if
$\nu_\eta((-\infty,0)) > 0$.
Further, since $\lim_{s\to\infty} f(s) = +\infty$ $P_\xi$-a.s.$(f)$,
for any $\varepsilon >0$ we conclude that
$$
\nu_f((-\varepsilon,\varepsilon) \setminus \{0\}) = \int_0^\infty \nu_\eta
((-e^{f(s)}\varepsilon,e^{f(s)}\varepsilon)\setminus \{0\}) ds= \infty
$$
provided that $\nu_\eta \not\equiv 0$.
This shows that $0 \in \supp \nu_f$, $P_\xi$-a.s.$(f)$.
It then follows from \cite[Theorem 24.10]{sato} that
$$
\supp \law(V_f) = \bR, \quad P_\xi-\mbox{a.s.}(f)
$$
if $A_\eta>0$, or if $\nu_\eta((0,\infty))> 0$ and $\nu_\eta((-\infty,0)) > 0$.\\
Hence in that case $P(V_f \in B |\xi=f) > 0$ $P_\xi$-a.s.$(f)$ for any open set $B\neq \emptyset$,
so that $P(V\in B) = \int P(V_f \in B|\xi = f) \, dP_\xi(f)>0$.
Thus $\supp \law(V) = \bR$, which shows (i).

To show (ii), suppose $\eta$ is of finite variation with drift $a$, and
$\nu_\eta((0,\infty))>0$ and $\nu_\eta((-\infty,0)) = 0$.
Then, for $P_\xi$-a.e. $f$, $V_f\geq a \int_0^\infty e^{-f(s)} \, ds> -\infty$ and hence $V_f$
is of finite variation.
It then follows from \cite[Theorem 3.15]{sato2007} that $V_f$ has drift
$a \int_0^\infty e^{-f(s)} \, ds$ and \cite[Theorem 24.10]{sato} gives
$$
\supp \law(V_f) = \left[ a \int_0^\infty e^{-f(s)} \, ds, \infty\right).
$$
Since $P(V \in B) = \int P(V_f \in B|\xi=f) \, dP_\xi(f)$,
the assertion (ii) follows from  Lemma~\ref{lem-support}.
Finally, (iii) follows from (ii) by replacing $\eta$ by $-\eta$.\qed
\end{proof}

The following result is now immediate.

\begin{corollary} \label{cor-positive-subordinator}
Let $\xi$ be a L\'evy process drifting to $+\infty$, and $\eta$ another L\'evy process,
independent of $\xi$ such that $\cL(\eta_1)\in D_\xi$.
Then $V=\int_0^\infty e^{-\xi_{s-}} \,d\eta_s\geq 0$ a.s. if and only if $\eta$ is a subordinator.
\end{corollary}

\begin{remark} \label{rem-support-counter}
{\rm (i) Let $\xi$ and $\eta$ be two independent L\'evy processes such that $V=\int_0^\infty e^{-\xi_{s-}} \, d\eta_s$
converges almost surely and consider the associated GOU process $(V_t)_{t\geq 0}$ defined by \eqref{GOUdef}. Then
it is easy to see that $V_n = A_n V_{n-1} + B_n$ for each $n\in \bN$, where $((A_n,B_n)^T)_{n\in \bN}$ is an i.i.d. sequence of
bivariate random vectors given by
$$(A_n,B_n)^T = \left(e^{-(\xi_n - \xi_{n-1})}, e^{-(\xi_n-\xi_{n-1})} \int_{(n-1,n]} e^{\xi_{s-} -\xi_{n-1}} \, d\eta_s\right)^T$$
(e.g. \cite[Lemma 6.2]{lindnermaller05}). Further, if $V_0$ is chosen to be independent of $(\xi,\eta)^T$, then
$(V_0,\ldots, V_{n-1})^T$ is independent of $((A_k,B_k)^T)_{k\geq n}$ for each $n$.
Since $\law(V)$ is the stationary marginal distribution of the GOU process, it is also the stationary marginal distribution of the random recurrence
equation $V_n = A_n V_{n-1} + B_n$, $n\in \bN$. We have seen in particular, that the support of $\law(V)$ was always an interval.
Hence it is natural to ask if stationary solutions to arbitrary random recurrence equations will always have an interval as its support. We will see that this is not the case.
To be more precise,
let $\left((A_n,B_n)^T\right)_{n\in \bN}$ be a given i.i.d. sequence of bivariate random vectors.
Suppose that $(X_n)_{n\in \bN_0}$ is a strictly stationary sequence which satisfies the
random recurrence equation
\begin{equation}
X_n = A_n X_{n-1} + B_n , \quad n\in \bN,\label{eq-rc1}
\end{equation}
such that $(X_0, \ldots, X_{n-1})$ is independent of $\left((A_k,B_k)^T\right)_{k\geq n}$
(provided that such a solution exists) for every $n\in \bN$.
Then the support of $\law(X_0)$ does not need to be an interval,
even if $A_n$ is constant and hence $A_n$ and $B_n$ are independent.
To see this, let $A_n = 1/3$ and let $(B_n)_{n\in \bZ}$ be an i.i.d. sequence
such that $P(B_n  =0) = P(B_n=2)=\frac12$. Then
\begin{equation} \label{eq-Cantor}
X_n = \sum_{k=0}^\infty 3^{-k} B_{n-k}, \quad n\in \bN_0,
\end{equation}
defines a stationary solution of \eqref{eq-rc1}, which is unique in distribution.
Obviously, the support of $\law(X_0)$ is given by the Cantor set
$$\left\{ \sum_{n=0}^\infty 3^{-n} z_n : z_n \in \{0,2\}, \; \forall\; n\in \bN_0\right\},$$
which is totally disconnected and not an interval.\\
(ii) The stationary solution constructed in \eqref{eq-Cantor} is a $1/3$-decomposable distribution (see \cite[Definition 64.1]{sato} for the definition).
By Proposition 6.2 in \cite{BehmeLindner13}, there exists a bivariate L\'evy process $(\xi,\eta)^T$
such that $\xi_t = (\log 3) N_t$ for a Poisson process $(N_t)_{t\geq 0}$ and such that
$$
\int_0^\infty e^{-\xi_{s-}} \, d\eta_s = \int_0^\infty 3^{-N_{s-}} \, d\eta_s
$$
has the same distribution as $X_0$ from \eqref{eq-Cantor}.
In particular, its support is not an interval. Hence a similar statement to Theorem \ref{thm-support}
does not hold under dependence.}
\end{remark}

\section{Closedness of the range}\label{S3}

This section is devoted to show that, as in the well-known case of a deterministic process $\xi$, the range $R_\xi=\Phi_\xi(D_\xi)$ is closed under weak convergence. On the contrary, closedness of $R_\xi$ under convolution does not hold any more as will be shown in Corollary \ref{cor-noconvolutions} below. \\
It will also follow that the inverse mapping $(\Phi_\xi)^{-1}$ is continuous, provided that $\Phi_\xi$ is injective. Recall that $\Phi_\xi$ is injective if, for instance, $\xi$ is spectrally negative
(cf. \cite[Theorem 5.3]{BehmeLindner13}).
Further, for any $\xi$ drifting to $+\infty$, $\Phi_\xi$ is always injective when restricted to positive measures $\cL(\eta_1)$ \cite[Remark 5.4]{BehmeLindner13}.
Thus, although $\Phi_\xi$ need not be continuous
(which follows by an argument similar to \cite[Example
7.1]{BehmeLindner13}), the inverse of $\Phi_\xi$ restricted to positive measures will turn out to be always
continuous.

We start with the following proposition, which shows that the mapping $\Phi_\xi$ is closed.

\begin{proposition} \label{prop-closed-mapping}
Let $\xi$ be a L\'evy process drifting to $+\infty$.
Then the mapping $\Phi_\xi$ is closed in the sense that if
$\law(\eta_1^{(n)}) \in D_\xi$, $\eta_1^{(n)} \stackrel{d}{\to}
\eta_1$ and $\Phi_\xi(\law(\eta^{(n)}_1)) \stackrel{w}{\to} \mu$ for
some random variable $\eta_1$ and probability measure $\mu$ as
$n\to\infty$, then $\law(\eta_1)\in D_\xi$ and
$\Phi_\xi(\law(\eta_1)) = \mu$.
\end{proposition}

\begin{proof}
For $n\in \bN$, let $W^{(n)}$ be a random variable such that
$$
W^{(n)} \stackrel{d}{=} \int_0^\infty e^{-\xi_{s-}} \,
d\eta_s^{(n)} \quad \mbox{and $W^{(n)}$ is independent of
$(\xi,\eta^{(n)})^T$},
$$
where $\eta^{(n)}$ is a L\'evy process
induced by $\eta_1^{(n)}$ independent of $\xi$.
Then the limit $\cL(\eta_1)$ is infinitely divisible by
\cite[Lemma 7.8]{sato}) and we can define $\eta$ as a
L\'evy process induced by $\eta_1$, independent of $\xi$. Let
$W$ be a random variable with distribution $\mu$, independent of
$(\xi,\eta)^T$.  The proof of
\cite[Theorem 7.3]{BehmeLindner13}, more precisely the part leading to Equation
(7.12) there, then shows that for every $t>0$ we have
$$
\left (e^{-\xi_t}, \int_0^t e^{\xi_{s-}} \, d\eta_s^{(n)}\right )^T
\stackrel{d}{\to} \left (e^{-\xi_t}, \int_0^t e^{\xi_{s-}}\, d\eta_s\right )^T,
\quad n\to\infty.
$$
Due to independence this yields
$$
\left (W^{(n)}, e^{-\xi_t}, \int_0^t e^{\xi_{s-}} \, d\eta_s^{(n)}\right )^T
\stackrel{d}{\to} \left (W,e^{-\xi_t}, \int_0^t e^{\xi_{s-}}\, d\eta_s\right )^T,
\quad n\to\infty,
$$
and since $\law(W^{(n)})$ is the invariant distribution of the GOU process driven by $(\xi,\eta^{(n)})^T$, this implies
$$
W^{(n)} \stackrel{d}{=} e^{-\xi_t} \left( W^{(n)} + \int_0^t
e^{\xi_{s-}} \, d\eta_s^{(n)} \right) \stackrel{d}{\to} e^{-\xi_t}
\left( W + \int_0^t e^{\xi_{s-}} \, d\eta_s\right), \quad n\to
\infty.
$$
Since also $W^{(n)} \stackrel{d}{\to} W$ as $n\to\infty$, this shows that
$$
W \stackrel{d}{=} e^{-\xi_t} \left( W + \int_0^t e^{\xi_{s-}} \,
d\eta_s \right)
$$
for any $t>0$,
so that $\mu=\law(W)$ is an invariant distribution of the GOU process driven by $(\xi,\eta)^T$. 
By \cite[Theorem 2.1]{lindnermaller05}, or alternatively \cite[Theorem 2.1~(a)]{behmelindnermaller11},  this shows that $\int_0^\infty
e^{-\xi_{s-}} \, d\eta_s$ converges a.s.,  i.e. $\law(\eta_1)\in
D_\xi$, and that $$\mu = \law(W) = \law\left( \int_0^\infty
e^{-\xi_{s-}} \, d\eta_s\right) = \Phi_\xi(\law(\eta_1)),$$ giving the claim.\qed
\end{proof}

In order to show that $R_\xi$ is closed, we shall first show in Proposition \ref{prop-tight} below that if a sequence
$(\Phi_\xi(\law(\eta_1^{(n)})))_{n\in \bN}$ is tight, then $(\eta_1^{(n)})_{n\in \bN}$ is tight.
To achieve this, observe first that as a consequence of
\cite[Lemma 15.15]{kallenberg} and Prokhorov's theorem, a sequence $(\law(\eta_1^{(n)}))_{n\in \bN}$ of infinitely divisible distributions
on $\bR$ with characteristic triplets $(\gamma_n, \sigma_n^2, \nu_n)$ is tight if and only if
$$
\sup_{n\in \bN} \left| \gamma_n + \int_{\bR} x \left( \frac{1}{1+x^2} - \mathds{1}_{|x|\leq 1}
\right) \, \nu_n(dx) \right| < \infty
$$
and the sequence $(\widetilde{\nu}_n)_{n\in \bN}$ of finite positive measures on $\bR$ with
$$
\widetilde{\nu}_n(dx) = \sigma_n^2 \, \delta_0(dx) + \frac{x^2}{1+x^2} \nu_n(dx)
$$
is weakly relatively compact (in particular, this implies that $\sup_{n\in \bN} \widetilde{\nu}_n(\bR) <\infty$).  Using Prokhorov's theorem for finite measures (e.g. \cite[Theorem 7.8.7]{AshDoleansDade}), it is easy to see that this is equivalent to
\begin{align}
\sup_{n\in \bN} \sigma_n^2  &<  \infty, \label{eq-tight-1}\\
\sup_{n\in \bN} \int_{[-1,1]} x^2 \, \nu_n(dx) & <  \infty, \label{eq-tight-2}\\
\sup_{n\in \bN} \nu_n( \bR \setminus [-r,r])  &<  \infty, \quad \forall\, r > 0, \label{eq-tight-3}\\
\lim_{r\to \infty} \sup_{n\in \bN} \nu_n(\bR \setminus [-r,r]) & =  0, \quad \mbox{and}
\label{eq-tight-4}\\
\sup_{n\in \bN} |\gamma_n| & <  \infty. \label{eq-tight-5}
\end{align}

The following lemma gives direct uniform estimates for $\mu([-r,r])$ in terms of the L\'evy measure
or Gaussian variance of an infinitely divisible distribution $\mu$ which will be needed to prove Proposition \ref{prop-tight}.

\begin{lemma} \label{lem-tight-1}
Let $\mu$ be an infinitely divisible distribution on $\bR$ with characteristic triplet
$(\gamma,\sigma^2,\nu)$.
For $\varepsilon \in (0,1)$ denote by $I_\varepsilon$ the set
$$I_\varepsilon := \{ z \in \bR: 1 - \cos z \geq \varepsilon \}.$$
Then for any $p \in (0,1)$ and $a > 0$, there is some $\varepsilon = \varepsilon(a,p) \in (0,1)$
such that
\begin{equation} \label{eq-tight-6}
\frac{ \lambda^1 ( I_\varepsilon \cap [-y,y])}{\lambda^1([-y,y])} \geq 1-p, \quad \forall\; y \geq a,
\end{equation}
where $\lambda^1$ denotes the Lebesgue measure on $\bR$.
For $\delta > 0$, denote by
$$
\|\nu\|_\delta := \nu (\bR \setminus [-\delta,\delta])
$$
the total mass of $\nu_{| \bR \setminus [-\delta,\delta]}$ and
$$
M(\nu) := \int_{[-1,1]} x^2 \, \nu(dx).
$$
Further, let $c> 0$ be a constant such that
$$
\cos (t)-1 \leq -c t^2, \quad \forall\,  t \in [-1,1].
$$
Then
\begin{align}
\mu ([-r,r])  &\leq  4(e^{-\varepsilon(\delta/r,p) \|\nu\|_\delta} (1-p) + p), \quad \forall\;
p\in (0,1), r,\delta > 0, \label{eq-tight-7}\\
\mu ([-r,r])  &\leq  1 -  \min \{ e^{-\|\nu\|_{2r}}, 1-e^{-\|\nu\|_{2r}/2} \},
 \quad \forall\; r > 0, \label{eq-tight-7b}\\
\mu([-r,r])  &\leq  2r   \int_{-1/r}^{1/r} e^{-M(\nu) ct^2} \, dt, \quad \forall\; r \geq 1,
\quad  \label{eq-tight-8}
\end{align}
and
\begin{align}
\mu([-r,r])  &\leq  2r \int_{-1/r}^{1/r} e^{-\sigma^2 t^2/2} \, dt, \quad \forall\, r > 0. \quad \quad \quad \quad \quad \quad \quad \quad \; { }
\label{eq-tight-9}
\end{align}
\end{lemma}

\begin{proof}
Equation \eqref{eq-tight-6} is clear.
Let $r>0$.
Then an application of \cite[Lemma 5.1]{kallenberg} shows
\begin{equation} \label{eq-tight-10}
\mu([-r,r]) \leq 2r \int_{-1/r}^{1/r} |\widehat{\mu}(t)| \, dt = 2r \int_{-1/r}^{1/r}
\exp \left( - \sigma^2 t^2/2 + \int_{\bR} (\cos (xt) -1) \nu(dx)  \right)\, dt
\end{equation}
which immediately gives \eqref{eq-tight-9}. Let $\delta >0$. Equation \eqref{eq-tight-7} is trivial
when $\|\nu\|_\delta = 0$, and for $\|\nu\|_\delta > 0$
observe that by \eqref{eq-tight-10} and Jensen's inequality we can estimate
\begin{align}
\mu([-r,r]) & \leq 2r \int_{-1/r}^{1/r} \exp \left( \int_{|x|>\delta} (\cos (xt) -1)
\|\nu\|_\delta \frac{\nu(dx)}{\|\nu\|_\delta}\right) \, dt \nonumber \\
& \leq  2r \int_{-1/r}^{1/r} \left(\int_{|x|>\delta} e^{(\cos (xt) -1) \|\nu\|_\delta } \,
\frac{\nu(dx)}{\|\nu\|_\delta}\right) dt \nonumber \\
& =  \int_{|x|>\delta} \left( \frac{2r}{|x|} \int_{-|x|/r}^{|x|/r} e^{(\cos z -1)
\|\nu\|_\delta} \, dz \right) \frac{\nu(dx)}{\|\nu\|_\delta} . \label{eq-tight-11}
\end{align}
By \eqref{eq-tight-6} we estimate for $|x|\geq \delta$ and $p\in (0,1)$ with
$\varepsilon = \varepsilon(\delta/r,p)$
\begin{align*}
\frac{2r}{|x|} \int_{-|x|/r}^{|x|/r}& e^{(\cos z -1) \|\nu\|_\delta} \, dz\\
& \leq  \frac{4}{\lambda^1 ([-\frac{|x|}{r}, \frac{|x|}{r}])} \left( e^{-\varepsilon \|\nu\|_\delta}
\lambda^1 \left([-\frac{|x|}{r}, \frac{|x|}{r}] \cap I_\varepsilon\right) + \lambda^1
\left([-\frac{|x|}{r}, \frac{|x|}{r}] \setminus I_\varepsilon\right) \right) \\
&\leq  4(e^{-\varepsilon \|\nu\|_\delta} (1-p) + p),
\end{align*}
which together with \eqref{eq-tight-11} results in \eqref{eq-tight-7}.
Similarly, \eqref{eq-tight-8} is trivial when $M(\nu) = 0$, while for $M(\nu) > 0$ define
the finite measure $\rho$ on $[-1,1]$ by $\rho(dx) = x^2 \nu(dx)$.
We then estimate with
\eqref{eq-tight-10} and Jensen's inequality, for $r\geq 1$,
\begin{align*}
\mu([-r,r]) & \leq  2r \int_{-1/r}^{1/r} \exp \left( \int_{[-1,1]} \frac{\cos (xt) -1}{x^2} M(\nu)
\frac{\rho(dx)}{M(\nu)}  \right)\, dt \\
& \leq  2r \int_{-1/r}^{1/r} \left(  \int_{[-1,1]} \exp \left(\frac{\cos (xt) -1}{x^2} M(\nu)\right)
\frac{\rho(dx)}{M(\nu)}\right) \, dt \\
& \leq 2r \int_{-1/r}^{1/r}  \left( \int_{[-1,1]} e^{-ct^2 M(\nu)}  \frac{\rho(dx)}{M(\nu)} \right)
\, dt ,
\end{align*}
which gives \eqref{eq-tight-8}.
Finally, let us prove Equation \eqref{eq-tight-7b}. This  is again trivial when $\|\nu\|_{2r}=0$, so assume
$\|\nu\|_{2r}> 0$. By symmetry, we can assume without loss of generality that
$$\nu((-\infty,-2r)) \geq \|\nu\|_{2r}/2  > 0.$$
Let $(X_t)_{t\geq 0}$ be a L\'evy process with $\law(X_1)=\mu$, and define
$$Y_t := \sum_{0 < s \leq t, \Delta X_s < -2r} \Delta X_s \quad \mbox{and} \quad Z_t := X_t - Y_t,\quad t \in \bR,$$
where $\Delta X_s := X_s - X_{s-}$ denotes the jump size of $X$ at time $s$. Then $(Y_t)_{t\geq 0}$ and $(Z_t)_{t\geq 0}$
are two independent L\'evy processes, and $(Y_t)_{t\geq 0}$ is a compound Poisson process with L\'evy measure $\nu_{|(-\infty,-2r)}$.
Denote by $(N_t)_{t\geq 0}$ the underlying Poisson process in $(Y_t)_{t\geq 0}$ which counts the number of jumps of $(Y_t)_{t\geq 0}$.
 Then
\begin{eqnarray*}
\mu(\bR \setminus [-r,r]) & = & P(|Y_1+Z_1|> r) \\
& \geq & P(|Z_1|> r, Y_1=0) + P(|Z_1|\leq r, Y_1 < -2r) \\
& = & P(|Z_1|>r) \, P(N_1=0) + P(|Z_1|\leq r)\, P(N_1\geq 1) \\
& = & P(|Z_1|>r) e^{-\nu((-\infty, -2r))} + (1-P(|Z_1|> r)) (1-e^{-\nu((-\infty,-2r))}) \\
& \geq & \min  \{ e^{-\nu((-\infty, -2r))}, 1-e^{-\nu((-\infty, -2r))}\} \\
& \geq &  \min \{ e^{-\|\nu\|_{2r}}, 1-e^{-\|\nu\|_{2r}/2} \},
\end{eqnarray*}
which implies \eqref{eq-tight-7b}.\qed
\end{proof}

The next result is the key step in proving closedness of $R_\xi$.

\begin{proposition} \label{prop-tight}
Let $\xi$ be a L\'evy process drifting to $+\infty$ and
$(\law(\eta^{(n)}_1))_{n\in\bN}$ be a sequence in $D_\xi$ such that
$(\mu_n := \Phi_\xi(\law(\eta^{(n)}_1)))_{n\in \bN}$ is tight. Then also
$(\eta_1^{(n)})_{n\in \bN}$ is tight.
\end{proposition}

\begin{proof}
Denote by $(\gamma_n,\sigma_n^2,\nu_n)$ the characteristic triplet of $\eta_1^{(n)}$.
We have to show that conditions \eqref{eq-tight-1} -- \eqref{eq-tight-5}
are satisfied. Let $n\in \bN$. Since $\int_0^\infty e^{-\xi_{s-}} \, d\eta_s^{(n)}$ converges almost surely and since $\eta^{(n)}$ and $\xi$ are independent, conditioning on $\xi=f$ shows that
$$\left( \left. \int_0^\infty e^{-\xi_{s-}} \, d\eta_s^{(n)} \right| \xi = f \right) = \int_0^\infty e^{-f(s-)} \, d\eta_s^{(n)}$$ for $P_\xi$-almost every $f\in D([0,\infty), \bR)$,
where the integral $\int_0^\infty e^{-f(s-)} \, d\eta_s^{(n)}$ converges almost surely for each such $f$. Further, since $\sup_{s\in [0,1]} |\xi_s| < \infty$ a.s. by the c\`adl\`ag paths of $\xi$, there are $0<D_1\leq 1 \leq D_2 < \infty$ such that 
$$P\left(D_1\leq e^{-\xi_{s}}\leq D_2 \; \forall\; s\in [0,1]\right) \geq 1/2.$$
Consequently there are some measurable sets $A_n\subset D([0,\infty),\bR)$ with $P_\xi(A_n) \geq 1/2$ such that
$$
D_1 \leq e^{-f(s)} \leq D_2 \quad \forall\; f\in A_n, \, s\in [0,1],
$$
and
$$
\int_0^\infty e^{-f(s-)}
\, d\eta_s^{(n)} \; \mbox{converges a.s.,\,\,$\forall f\in A_n$}.
$$
Further, we obtain
\begin{eqnarray}
\mu_n (\bR \setminus [-r,r])
& \geq &   \int_{A_n} P\left(\left|\int_0^\infty e^{-f(s-)}\, d\eta_s^{(n)}\right| > r\right) P_\xi(df) \nonumber \\
& \geq & \frac12 \left( 1 - \sup_{f\in A_n} P\left(  \left|\int_0^\infty e^{-f(s-)}\, d\eta_s^{(n)}\right| \leq r\right)\right) . \label{eq-tight-12}
\end{eqnarray}
For fixed $f\in A_n$ the distribution of $\int_0^\infty e^{-f(s-)} \, d\eta_s^{(n)}$ is
infinitely divisible with Gaussian variance
\begin{equation} \label{eq-tight-13}
\sigma_{f,n}^2 = \sigma_n^2 \int_0^\infty (e^{-f(s)})^2 \, ds \geq D^2_1 \sigma_n^2\end{equation}
and L\'evy measure $\nu_{f,n}$ satisfying
\begin{equation} \label{eq-measure-induction}
\nu_{f,n} (B) = \int_0^{\infty} ds \int_{\bR} \mathds{1}_B (e^{-f(s)} x) \, \nu_n(dx)
\end{equation}
for any Borel set $B\subset \bR\setminus \{0\}$ (cf. \cite[Theorem 3.10]{sato2007}).
In particular, for $f\in A_n$ and any $\delta > 0$,
\begin{eqnarray}
\nu_{f,n} (\bR \setminus [-\delta, \delta]) & \geq& \int_0^1ds \int_{\bR} \mathds{1}_
{\bR \setminus [-\delta e^{f(s)}, \delta  e^{f(s)}]} (x) \, \nu_n (dx) \, \nonumber \\
 & \geq & \nu_n (\bR \setminus
[-\delta/D_1,\delta/D_1]). \quad { } \label{eq-tight-14}
\end{eqnarray}
From \eqref{eq-measure-induction} we obtain
$$\int_{[-1,1]} t^2 \nu_{f,n} (dt)   =
\int_0^{\infty}ds \int_{\bR} \left(e^{-f(s)}x\right)^2 \mathds{1}_{\{|e^{-f(s)}x|\leq 1\}}(x)\,
\nu_n(dx),$$
for $f\in A_n$, hence
\begin{eqnarray}
\int_{[-1,1]} t^2 \nu_{f,n} (dt)
 &\geq& D_1^2  \int_{[-1,1]} x^2 \mathds{1}_{\{|D_2 x| \leq 1\}}(x) \nu_n(dx) \nonumber \\
& = & D_1^2 \int_{[-1/D_2,1/D_2]} x^2 \nu_n(dx). \label{eq-tight-15}
\end{eqnarray}
Now suppose \eqref{eq-tight-1} were violated. Then by \eqref{eq-tight-12},
\eqref{eq-tight-9} and \eqref{eq-tight-13}  we conclude that
$$
\sup_{n\in \bN} \big\{ \mu_n (\bR\setminus [-r,r]) \big\} \geq \frac12 \sup_{n\in \bN}
\left\{  1 -   2r  \int_{-1/r}^{1/r} e^{-D_1^2 \sigma_n^2 t^2/2} \, dt \right\} = \frac12
$$
for every $r>0$, contradicting tightness of  $(\mu_n)_{n\in \bN}$.
Hence \eqref{eq-tight-1} must be true.

Now suppose that \eqref{eq-tight-3} were violated, so that there is some $\delta>0$
such that $\sup_{n\in \bN} \|\nu_n\|_\delta = \infty$ with the notions of Lemma~\ref{lem-tight-1}.
Let $p\in (0,1/4)$ be arbitrary.
Then by \eqref{eq-tight-7} and \eqref{eq-tight-14}, we have for every $f\in A_n$, with $\epsilon=\epsilon(D_1\delta/r,p)$ as defined
in Lemma~\ref{lem-tight-1}, that
\begin{eqnarray}
P\left( \left| \int_0^\infty e^{-f(s-)} \, d\eta_s^{(n)} \right| \leq r\right) & \leq &
4 \left( e^{-\varepsilon (D_1 \delta/r,p) \|\nu_{f,n}\|_{D_1\delta}} (1-p) + p\right) \label{eq-concentration} \\
& \leq & 4  e^{-\varepsilon (D_1 \delta/r,p) \|\nu_{n}\|_{\delta}}(1-p) + 4p. \nonumber
\end{eqnarray}
From \eqref{eq-tight-12} we then obtain that
$$
\sup_{n\in \bN} \big\{ \mu_n (\bR\setminus [-r,r]) \big\}\geq
\frac12 (1-4p)>0, \quad \forall\; r > 0,
$$
which again contradicts tightness of $(\mu_n)_{n\in \bN}$ so that \eqref{eq-tight-3} must hold.

Now suppose that \eqref{eq-tight-4} were violated.
Then there is
some $a>0$ and a sequence $(\delta_k)_{k\in \bN}$ of positive real numbers tending to $+\infty$
and an index $n(k)\in \bN$ for each $k$ such that
$$
\|\nu_{n(k)}\|_{2\delta_k/D_1} \geq a, \quad \forall\, k \in \bN.
$$
Let $p\in (0,1/4)$ be arbitrary and choose $\varepsilon = \varepsilon(D_1,p)$ as in Lemma~\ref{lem-tight-1}. Let $b>0$
be such that
$$b_1 := 4 \left(  e^{-\varepsilon(D_1,p) b} (1-p) + p\right) < 1.$$
Let $f\in A_n$. Then if $\|\nu_{f,n(k)}\|_{D_1\delta_k} \geq b$ we have
$$P\left( \left| \int_0^\infty e^{-f(s-)} \, d\eta_s^{(n(k))} \right| \leq \delta_k\right) \leq b_1 <1$$
by \eqref{eq-concentration}, while if $\|\nu_{f,n(k)}\|_{D_1 \delta_k} < b$ we obtain from \eqref{eq-tight-7b} and \eqref{eq-tight-14} that
\begin{eqnarray*}
P\left( \left| \int_0^\infty e^{-f(s-)} \, d\eta_s^{(n(k))} \right| \leq \delta_k\right) & \leq & 1 - \min \{ e^{-\|\nu_{f,n(k)}\|_{2\delta_k}},
 1- e^{-\|\nu_{f,n(k)}\|_{2\delta_k}/2} \}\\
 & \leq & 1- \min \{e^{-b}, 1 - e^{-\|\nu_{n(k)}\|_{2\delta_k/D_1}/2} \} \\
 & \leq & 1- \min\{ e^{-b}, 1-e^{-a/2} \}.
 \end{eqnarray*}
From \eqref{eq-tight-12} we then conclude
$$
\mu_{n(k)} (\bR \setminus [-\delta_k,\delta_k]) \geq \frac12 \left( 1- \max \{ b_1, 1-e^{-b}, e^{-a/2} \}\right) >0 \quad \forall\, k\in \bN.
$$
In particular,
$$
\limsup_{r\to\infty} \sup_{n\in \bN} \left\{ \mu_n (\bR \setminus [-r,r])\right\}
\geq \frac12 \left( 1- \max \{ b_1, 1-e^{-b}, e^{-a/2} \}\right) >0,
$$
which again contradicts tightness of $(\mu_n)_{n\in \bN}$.
We conclude that also \eqref{eq-tight-4} must be valid.

Now suppose that \eqref{eq-tight-2} were violated, but \eqref{eq-tight-3} holds.
Then by \eqref{eq-tight-12}, \eqref{eq-tight-8}, \eqref{eq-tight-15} and with $c$ from Lemma~\ref{lem-tight-1} we have for every $r\geq 1$
\begin{eqnarray*}
\lefteqn{\sup_{n\in \bN} \big\{ \mu_n (\bR\setminus [-r,r]) \big\}}\\ &\geq& \frac12 \sup_{n\in \bN}
\left\{  1 -   2r  \int_{-1/r}^{1/r} \exp \left( -D_1^2 c t^2 \int_{[-1/D_2,1/D_2]} x^2 \,
\nu_n(dx) \right) dt \right\} = \frac12,
\end{eqnarray*}
where we have used that \eqref{eq-tight-3} together
with  $\sup_{n\in \bN} \int_{[-1,1]} x^2 \, \nu_n(dx) = \infty$ imply \linebreak
$\sup_{n\in \bN} \int_{[-1/D_2,1/D_2]} x^2 \, \nu_n(dx) = \infty$.
This again contradicts tightness of $(\mu_n)_{n\in \bN}$
so that \eqref{eq-tight-2} must hold.

Finally, suppose  that \eqref{eq-tight-5} were violated but that \eqref{eq-tight-1}--\eqref{eq-tight-4} hold. Then there is a subsequence of
$(\gamma_n)_{n\in \bN}$
which diverges to $+\infty$ or $-\infty$, and without loss of generality assume that this is $(\gamma_n)_{n\in \bN}$.
Since $(\mu_n)_{n\in \bN}$ is tight by assumption,
there is a subsequence of $(\mu_n)_{n\in \bN}$ which converges weakly,
and for the convenience of notation assume again that $(\mu_n)_{n\in \bN}$ converges
weakly to some distribution $\mu$.
Let the L\'evy process $U$ with characteristic triplet
$(\gamma_U,\sigma_U^2, \nu_U)$ be related to $\xi$ by $\cE(U)_t = e^{-\xi_t}$, where $\cE(U)$ denotes the stochastic exponential of $U$.
Then it follows from \cite[Corollary 3.2 and Equation (4.1)]{BehmeLindner13} that
\begin{align*}
\gamma_n \int_{\bR} f'(x)  \, \mu_n(dx) 
& = -\frac12 \sigma_n^2 \int_{\bR} f''(x) \, \mu_n(dx)  \\
& \hskip 5mm -\int_{\bR}\mu_n(dx)
\int_{\bR} \left( f(x+y) -f(x) -
f'(x) y\mathds{1}_{|y|\leq 1} \right) \nu_n(dy)   \\
& \hskip 5mm - \gamma_U \int_{\bR} f'(x) x \, \mu_n(dx) -
\frac12 \sigma_U^2 \int_{\bR} f''(x) x^2 \, \mu_n(dx) \\
& \hskip 5mm  - \int_{\bR}\mu_n(dx) \int_{\bR} \left( f(x+xy) - f(x) - f'(x)
xy \mathds{1}_{|y|\leq 1}\right)
\, \nu_U(dy)
\end{align*}
for every function $f\in C_c^2(\bR)$. Consider the right hand side of this equation.
The first summand remains bounded in $n$ by \eqref{eq-tight-1} and weak convergence of $\mu_n$,
and the second remains bounded in $n$ by \eqref{eq-tight-2} and \eqref{eq-tight-3}, since
$$
|f(x+y) - f(x) - f'(x) \mathds{1}_{|y|\leq 1}| \leq 2 \|f\|_\infty \mathds{1}_{|y|>1}
+ \|f''\|_\infty y^2 \mathds{1}_{|y|\leq 1}
$$
(cf. \cite[Proof of Lemma 4.2]{BehmeLindner13}), where $\|\cdot\|_\infty$ denotes the supremum norm.
The third and fourth summands converge by weak convergence of $\mu_n$, and the fifth
summand remains bounded in $n$ by \cite[Equation (3.6)]{BehmeLindner13} (actually, the fifth summand
can be shown to converge).
We conclude  also that $\gamma_n \int_{\bR} f'(x) \mu_n(dx)$ must be bounded in $n$ for every
$f\in C_c^2(\bR)$.
Choosing $f\in C_c^2(\bR)$ such that $\int_{\bR} f'(x) \, \mu(dx) \neq 0$,
we obtain that $(\gamma_n)_{n\in \bN}$ must be bounded and hence the desired contradiction.
Summing up, we have verified \eqref{eq-tight-1} -- \eqref{eq-tight-5} so that
$(\eta_1^{(n)})_{n\in \bN}$ must be tight.\qed
\end{proof}

Now define
\begin{align*}
D_\xi^+ & :=  \{ \law(\eta_1) \in D_\xi : \eta_1 \geq 0 \;
\mbox{a.s.} \},\\
\Phi_\xi^+ & :=  (\Phi_\xi)|_{D_\xi^+},
\end{align*}
and
\begin{align*}
R_\xi^+  &:=  \Phi_\xi (D_\xi^+) = \Phi_\xi^+(D_\xi^+). \quad \quad \; \, \,  { }
\end{align*}
By Corollary \ref{cor-positive-subordinator},
$$
R_\xi^+ = R_\xi \cap \{ \mu \in \mathcal{P}(\bR): \supp \mu \subset [0,\infty) \}.
$$

We now show closedness of $R_\xi$ under weak convergence and that the inverse of $\Phi_\xi$ (provided that
it exists) is continuous.

\begin{theorem} \label{thm-cont-inverse1}
Let $\xi=(\xi_t)_{t\geq 0}$ be a L\'evy process drifting to
$+\infty$.
\begin{enumerate}
 \item Then $R_\xi$ and $R_\xi^+$ are closed under weak convergence.
\item If $\Phi_\xi$ is injective,
then the inverse $\Phi_\xi^{-1}:R_\xi \to D_\xi$ is
continuous with respect to the topology induced by weak convergence.
\item The inverse  $(\Phi_\xi^+)^{-1}:R_\xi^+ \to
D_\xi^+$ is continuous.
\end{enumerate}
\end{theorem}

\begin{proof}
(i) Let $(\mu_n= \Phi_\xi(\law(\eta_1^{(n)})))_{n\in \bN}$ be a sequence in $R_\xi$
which converges weakly to some
$\mu\in \cP(\bR)$. Then $(\mu_n)_{n\in \bN}$ is tight, and by
Proposition \ref{prop-tight}, $(\eta_1^{(n)})_{n\in \bN}$ must be tight, too.
Hence there is a subsequence $(\eta_1^{(n_k)})_{k\in \bN}$ which converges
weakly to some random variable $\eta_1$. It then follows from
Proposition \ref{prop-closed-mapping} that also $\law(\eta_1) \in D_\xi$ and
that $\Phi_\xi(\law(\eta_1)) = \mu$. Hence $\mu\in R_\xi$ so that $R_\xi$ is closed.
Since $\{ \mu \in \mathcal{P}(\bR): \supp \mu \subset [0,\infty) \}$ is closed,
this gives also closedness of $R_\xi^+$.

(ii) Let $(\mu_n= \Phi_\xi(\law(\eta_1^{(n)})))_{n\in \bN}$ be a sequence in $R_\xi$ which converges weakly to some $\mu$.
By Proposition \ref{prop-tight}, $(\eta_1^{(n)})_{n\in \bN}$ is tight.
Let $(\eta_1^{(k_n)})_{k\in \bN}$ be a subsequence which converges weakly to some $\eta_1$, say.
Then $\law(\eta_1) \in D_\xi$ and $\Phi_\xi (\law(\eta_1)) = \mu$ by
Proposition \ref{prop-closed-mapping}, and since $\Phi_\xi$ is injective we have
$\law(\eta_1) = \Phi_\xi^{-1} (\mu)$. Since the convergent subsequence was arbitrary,
this shows that $\law(\eta_1^{(n)}) = \Phi_\xi^{-1} ( \mu_n)$ converges weakly to
$\Phi_\xi^{-1}(\mu)$ as $n\to \infty$ (cf. \cite[Corollary to Theorem 25.10]{Billingsley1995}).
Hence $\Phi_\xi$ is continuous.

(iii) This can be proved in complete analogy to (ii).\qed
\end{proof}

\begin{remark} \label{rem-closed-positive}
{\rm Closedness of $R_\xi^+$ under weak convergence and continuity of $(\Phi_\xi^+)^{-1}$
can also be proved in a simpler way by circumventing Proposition \ref{prop-tight} but
using a formula for the Laplace transforms of
$\eta_1^{(n)}$ and $\mu_n$ (cf. \cite[Remark 4.5]{BehmeLindner13}, or Theorem \ref{thm-range-condition} below), and showing that
$\mu^{(n)} \stackrel{w}{\to} \mu$ implies convergence of the Laplace transforms of $\eta_1^{(n)}$.
A similar approach for showing closedness of $R_\xi$ is not evident since there is not a
similarly convenient formula for the Fourier transforms available, but only one in
terms of suitable two-sided
limits (cf. \cite[Equation (4.7)]{BehmeLindner13}).}
\end{remark}

As a consequence of Theorem \ref{thm-cont-inverse1}, we can now show that $R_\xi$ will not be closed
under convolution if $\xi$ is non-deterministic and satisfies a suitable moment condition.
We conjecture
that $R_\xi$ will never be closed under convolution unless $\xi$ is deterministic.

\begin{corollary} \label{cor-noconvolutions}
Let $\xi=(\xi_t)_{t\geq 0}$ be a non-deterministic L\'evy process drifting to $+\infty$ such that
$E[(e^{-2\xi_1})]<1$.
Then $R_\xi$ is not closed under convolution.
\end{corollary}
\begin{proof}
Let $(\eta_t)_{t\geq 0}$ be a symmetric compound Poisson process with L\'evy measure
$\nu=\delta_{-1}+\delta_1$, where $\delta_a$ denotes the Dirac measure at $a$.
Then $\cL(\eta_1)\in D_\xi$ and $V:=\int_0^\infty e^{-\xi_{s-}}d\eta_s$ is symmetric, too,
and since by \cite[Theorem 3.3]{behme2011} we have $E[V^2]<\infty$, this yields $E[V]=0$.
Now let $(V_i)_{i\in\NN}$ be an i.i.d. family of independent copies of $V$.
Then by the Central Limit Theorem,
$$\cL\left(n^{-\frac{1}{2}}(V_1+\ldots+ V_n)\right) \to \cN (0,\var(V)), \quad n\to \infty,$$
with $\var(V)\neq 0$.
If the range $R_\xi$ was closed under convolution, we consequently had $\cL(n^{-\frac{1}{2}}(V_1+\ldots + V_n))
\in R_\xi$ and due to closedness of $R_\xi$ under weak convergence
this gave $\cN (0,\var(V))\in R_\xi$.
This contradicts \cite[Theorem 6.4]{BehmeLindner13}.\qed
\end{proof}

\section{A general criterion for a positive distribution to be in the range} \label{S4}

From this section on, we restrict ourselves to positive distributions in $R_\xi$, i.e. we only consider $\Phi_\xi^+$ and $R_\xi^+$ as defined in Section \ref{S3}. We start by giving a general criterion to decide whether a positive distribution is in the range $R_\xi^+$  of $\Phi_\xi^+$ for a given L\'evy process $\xi$.

\begin{theorem} \label{thm-range-condition}
Let $\xi$ be a L\'evy process drifting to $+\infty$.
Let $\mu = \law (V)$  be a probability measure on
$[0,\infty)$ with Laplace exponent $\psi_V$.
Then $\mu \in R_\xi^+$ if and only if the function
\begin{align}
g_\mu &: (0,\infty) \to  \bR \nonumber\\
g_\mu(u)& :=  (\gamma_\xi - \frac{\sigma_\xi^2}{2}) u \psi_V'(u) - \frac{\sigma_\xi^2}{2} u^2
\left( \psi_V''(u) + (\psi_V'(u))^2\right) \nonumber \\
&  \hskip 10mm - \int_{\bR} \left( e^{\psi_V(ue^{-y}) - \psi_V(u)} - 1 + u \psi_V'(u) y
\mathds{1}_{|y|\leq 1} \right) \, \nu_\xi(dy), \quad u >0, \label{eq-range-cond-0}
\end{align}
defines the Laplace exponent of some subordinator $\eta$, i.e.
if there is some subordinator $\eta$  such that
\begin{equation} \label{eq-range-cond-1}
E \left[e^{-\eta_1 u}\right] = e^{g_\mu(u)}, \quad \forall\; u > 0.
\end{equation}
In that case, $\Phi_\xi(\law (\eta_1)) = \mu$.
\end{theorem}

Using a Taylor expansion for $|y|\leq 1$, it is easy to see that the integral defining
$g_\mu$  converges for every distribution $\mu$ on $[0,\infty)$.

\begin{proof}
Observe first that
\begin{align}
- E\left[V e^{-uV} \right]&= \LL_V'(u)=  \psi_V'(u)e^{\psi_V(u)} \label{eq-derivative-1}\\
E\left[V^2 e^{-uV} \right]&= \LL_V''(u)=
\psi_V''(u)e^{\psi_V(u)}+(\psi_V'(u))^2 e^{\psi_V(u)} \label{eq-derivative-2}
\end{align}
for $u>0$. Hence
\begin{align}
g_\mu (u) & \LL_{V} (u)\nonumber\\
& = -u \gamma_\xi E \left[ V e^{-u V} \right] - \frac{\sigma_\xi^2}{2} \left( E \left[ V^2 e^{-uV}\right] u^2 - E\left[V e^{-uV}\right] u  \right)\label{eq-februar2014}\\
& \hskip 8mm - \int_{(-1,\infty)}\left( \LL_{V} (u e^{-y}) - \LL_{V} (u)
- u E \left[ V e^{-u V} \right] y \mathds{1}_{|y|\leq 1}
\right) \nu_\xi(dy), \;\; \forall\; u > 0. \nonumber
\end{align}

Now if $\mu = \law(V) \in R_\xi^+$, let $\law(\eta_1) \in D_\xi^+$ such that  $\mu=\law(V) = \Phi_\xi(\law(\eta_1))$. Then $g_\mu = \log \LL_\eta$
by Remark~4.5 of \cite{BehmeLindner13}, so that  \eqref{eq-range-cond-1} is satisfied.

Conversely, suppose that $V \geq 0$, and let $\eta$
be a subordinator such that
\eqref{eq-range-cond-1} is true. Define the L\'evy process $U$ by $e^{-\xi_t}  =\cE(U)_t$, where $U$ denotes the stochastic exponential of $U$. Then by \cite[Remark 4.5]{BehmeLindner13} and \eqref{eq-februar2014},
\eqref{eq-range-cond-1} is equivalent to
\begin{align*}
\log & \LL_\eta (u) \, \LL_{V} (u) \\
& = u \gamma_U E \left[ V e^{-u V} \right] - \frac{\sigma_U^2 u^2}{2} E \left[ V^2 e^{-uV}\right] \\
& \hskip 8mm - \int_{(-1,\infty)}\left( \LL_{V} (u (1+y)) - \LL_{V} (u)
+ u E \left[ V e^{-u V} \right] y \mathds{1}_{|y|\leq 1}
\right) \nu_U(dy), \;\; \forall\; u > 0,
\end{align*}
and a direct computation using \eqref{eq-derivative-1} and \eqref{eq-derivative-2}
 shows that this in turn is equivalent to
\begin{align}
0 & =  \int_{[0,\infty)} \left( f'(x) (x \gamma_U + \gamma_\eta^0) + \frac12 f''(x) x^2
\sigma_U^2\right) \, \mu(dx) \nonumber\\
&  \hskip 10mm + \int_{[0,\infty)}\mu(dx) \int_{(-1,\infty)} \left( f(x+xy) - f(x) - f'(x) xy
\mathds{1}_{|y|\leq 1}\right) \nu_U(dy)   \nonumber \\
& \hskip 10mm + \int_{[0,\infty)}\mu(dx) \int_{[0,\infty)} \left( f(x+y) - f(x) \right) \, \nu_\eta(dy)
  \label{eq-range-cond-2}
\end{align}
for all functions $f \in \mathcal{G} := \{ h \in C^2_b(\bR) : \exists \, u > 0 \;\mbox{such that}\; h(x) = e^{-ux}, \; \forall x\geq 0 \}$, where $\gamma_\eta^0$ denotes the drift of $\eta$.
Observe that \eqref{eq-range-cond-2} is also trivially true for $f\equiv 1$. Denote by
\begin{align*}
{\bf R} [\mathcal{G}]  :=  \big\{ h \in C^2_b(\bR) :  & \,\, \exists n\in \bN_0, \, \exists \lambda_1,
\ldots, \lambda_n \in \bR,
\, \exists u_1,\ldots, u_n \geq 0\\
& \mbox{such that}\,\,
 h(x) = \sum_{k=1}^n \lambda_k e^{-u_k x} \; \forall\; x \geq 0 \big\}
\end{align*}
the algebra generated by $\mathcal{G}$.
By linearity, \eqref{eq-range-cond-2} holds true
also for all $f\in {\bf R}[\mathcal{G}]$. Since $\mathcal{G}$
is strongly separating and since for each $x\in \bR$ there exists $h\in \mathcal{G}$
such that $g'(x) \neq 0$, the set $\mathcal{G}$ satisfies
condition (N) of \cite[Definition 1.4.1]{Llavona}, and hence ${\bf R}[\mathcal{G}]$ is dense
in $\mathcal{S}^2(\bR)$ by \cite[Corollary. 1.4.10]{Llavona}, where
$$
\mathcal{S}^2(\bR) := \{ h \in C^2(\bR): \lim_{|x|\to \infty} (1+|x|^2)^k
(|h(x)| + |h'(x)| + |h''(x)|) = 0, \,\, \forall\, k\in \bN_0\}
$$
is the space of rapidly decreasing functions of order 2, endowed with the usual topology
(cf. \cite[Definition 0.1.8]{Llavona}). In particular, for every $f\in C_c^2(\bR)\subset \mathcal{S}^2(\bR)$
there exists a sequence $(f_n)_{n\in \bN}$ in
${\bf R}[\mathcal{G}]$ such that
$$
\lim_{n\to\infty} \sup_{x\in \bR} \left[ (1+|x|^2) \left( |f_n(x) - f(x)| + |f_n'(x) - f'(x)|
+ |f_n''(x) - f''(x)| \right)\right] = 0.
$$
Since \eqref{eq-range-cond-2} holds for each $f_n$, an application of Lebesgue's dominated
convergence theorem shows
that \eqref{eq-range-cond-2} also holds for $f\in C_c^2(\bR)$; remark that
Lebesgue's theorem can be applied by Equations (3.5) and (3.6) in
\cite{BehmeLindner13} for the integral with respect to $\nu_U$ and $\mu$ and by observing that
$$
|f(x+y) - f(x)| \leq 2 \|f\|_\infty \mathds{1}_{y>1} + \|f'\|_\infty y \mathds{1}_{0<y\leq 1}
$$
for the integral with respect to $\nu_\eta$ and $\mu$.

Since $C_c^2(\bR)$ is a core
for the Feller process
\begin{equation} \label{eq-GOU-diff}
W_t^x = x + \int_0^t W_{s-}^x \, dU_s + \eta_t
\end{equation}
with generator
\begin{align*}
 A^W f(x)&=f'(x)(x\gamma_U+\gamma_\eta^0) + \frac{1}{2} f''(x)x^2\sigma_U^2 \\
& \hskip 10mm + \int_{(-1,\infty)} (f(x+xy) - f(x) - f'(x)xy \mathds{1}_{|y|\leq 1}) \nu_U(dy) \\
& \hskip 10mm + \int_{[0,\infty)} (f(x+y) - f(x) ) \nu_{\eta}(dy)
\end{align*}
for $f\in C_c^2(\bR)$ (cf. \cite[Theorem 3.1 and Corollary 3.2]{BehmeLindner13} and
\cite[Equation (8.6)]{sato}), we have that
$\int_{\bR} A^W f(x) \, \mu(dx) = 0$ for all $f$ from a core, and hence $\mu=\law(V)$
is an invariant measure for the GOU process \eqref{eq-GOU-diff}  by
\cite[Theorem 3.37]{liggett}.
By \cite[Theorem 2.1(a)]{behmelindnermaller11},
this implies that $\int_0^\infty e^{-\xi_{s-}} \, d\eta_s$ converges a.s. and that
$ \law(\int_0^\infty e^{-\xi_{s-}} \, d\eta_s) = \mu$, so that $\law(\eta_1) \in D_\xi^+$
and $\Phi_\xi (\law(\eta_1)) = \mu$, completing the proof.\qed
\end{proof}

\begin{remark} \label{rem-BehmeSchnurr}
{\rm To obtain a similar handy criteria for a non-positive distribution to be in the range $D_\xi$
seems harder.
A general \emph{necessary} condition in this vein for a distribution $\mu=\law(V)$ to be in the
range $R_\xi$, where $\xi$ is a L\'evy process with characteristic triplet
$(\gamma_\xi, \sigma_\xi^2, \nu_\xi)$, can be derived from Equation (4.7)
in \cite{BehmeLindner13}.
If further $E [V^2] < \infty$, and $\log \phi_\eta (u)$ denotes the characteristic
exponent of a L\'evy process $\eta$ such that $E [e^{iu\eta_1}] ={\phi}_\eta(u)$, $u\in \bR$, then by Equation (4.8) in \cite{BehmeLindner13},
\begin{align}
\phi_V(u) \log {\phi}_\eta (u)  & = \gamma_\xi u \phi_V'(u) - \frac{\sigma_\xi^2}{2}
\left( u^2 \phi_V''(u) + u \phi_V'(u)\right) \nonumber \\
&  \hskip 10mm -\int_{\bR} (\phi_V (u e^{-y}) - \phi_V(u) + uy\phi_V'(u)
\mathds{1}_{|y|\leq 1}) \, \nu_\xi(du). \label{eq-BehmeSchnurr}
\end{align}
In \cite[Example 3.2]{BehmeSchnurr13}, this equation has been derived using the theory of symbols.
Hence, a necessary condition for $V$ with $E [V^2] <\infty$ to be in $R_\xi$ is that there is a
L\'evy process $\eta$, such that the right-hand side of \eqref{eq-BehmeSchnurr} can be expressed
as $\phi_V(u) \log{\phi}_\eta(u) $, $u\in \bR$.
In Example 4.3 of \cite{BehmeSchnurr13} it has been shown that the existence of some
L\'evy process $\eta$ such that the right-hand side of \eqref{eq-BehmeSchnurr} can be expressed
as $\phi_V(u) \log {\phi}_\eta(u) $ is also \emph{sufficient} for $\mu = \law(V)$ with
$E [V^2] < \infty$ to be in $R_\xi$, hence this is a \emph{necessary and sufficient} condition
for $\law(V)$ with $E [V^2] < \infty$ to be in $R_\xi$, similar to
Theorem \ref{thm-range-condition}.
Without the assumption $E V^2 <\infty$, a necessary and sufficient condition is not
established at the moment.}
\end{remark}

We conclude this section with the following results:

\begin{lemma} \label{lemma-no-drift}
Let $\xi$ be a  spectrally negative L\'evy process of infinite variation, drifting to $+\infty$.
Then every element in $R_\xi^+$ is selfdecomposable and of finite variation with drift $0$.
\end{lemma}

\begin{proof}
That any element in $R_\xi^+$ must be selfdecomposable has been shown in
\cite{bertoinlindnermaller08}, since $\xi$ is spectrally negative.
Since every element in $R_\xi^+$ is positive, it must be of finite variation,
and it follows from Theorem \ref{thm-support} and \cite[Theorem 24.10]{sato} that the drift must be 0.\qed
\end{proof}

\begin{remark} \label{rem-no-compound-Poisson}
{\rm It is well known that a selfdecomposable distribution cannot have finite non-zero
L\'evy measure, in particular it cannot be a compound Poisson distribution, which follows
for instance immediately from \cite[Corollary 15.11]{sato}.
This applies in particular to exponential functionals of L\'evy processes
with spectrally negative $\xi$.
However, even if $\xi$ is not spectrally negative, and $(\xi,\eta)^T$ is a bivariate
(possibly dependent) L\'evy process, then $\int_0^\infty e^{-\xi_{s-}} \, d\eta_s$ (provided it converges) still cannot
be a non-trivial compound Poisson distribution, with or without drift. For
if $c$ denotes the drift of a non-trivial compound Poisson distribution with drift,
then this distribution must have an atom at $c$. However, e.g. by
\cite[Theorem 2.2]{bertoinlindnermaller08}, $\law (\int_0^\infty e^{-\xi_{s-}} \, d\eta_s)$
must be continuous unless constant. In other words, if $\int_0^\infty e^{-\xi_{s-}} \, d\eta_s$ is infinitely divisible, non-constant and has no Gaussian part, then its L\'evy measure must be infinite.
In particular, it follows that if $\eta$ is a subordinator
and $\int_0^\infty e^{-\xi_{s-}} \, d\eta_s$ is infinitely divisible and non-constant, then its L\'evy measure must be infinite.}
\end{remark}

\section{Some results on $R_\xi^+$ when $\xi$ is a Brownian motion}\label{S5}

It is particularly interesting to study the distributions $\int_0^\infty e^{-\xi_{s-}} \, d\eta_s$ when one of the independent L\'evy processes $\xi$ or $\eta$ is a Brownian motion with drift. While the paper \cite{savovetal} focuses on  the case when $\eta$ is a Brownian motion with drift,
in this section we specialise to the case $\xi_t=\sigma B_t + at$, $t\geq 0$, 
with $\sigma,a>0$ and $(B_t)_{t\geq 0}$ a standard
Brownian motion. Then by Lemma~\ref{lemma-no-drift}, $R_\xi^+$ is a subset of $L(\bR_+)$,
the class of selfdecomposable distributions on $\bR_+$. Recall that a distribution
$\mu=\law(V)$ on $\bR_+$ is selfdecomposable if and only if it is infinitely divisible
with non-negative drift and its L\'evy measure has a L\'evy density of the form
$(0,\infty) \to [0,\infty)$, $x\mapsto x^{-1} k(x)$ with a non-increasing function
$k=k_V:(0,\infty) \to [0,\infty)$ (cf. \cite[Corollary 15.11]{sato}).
Further, to every distribution $\mu=\law(V) \in L(\bR_+)$ there exists a subordinator
$X=(X_t)_{t\geq 0}=(X_t(\mu))_{t\geq 0}$, unique in distribution, such that
\begin{equation} \label{eq-X-mu}
\mu = \cL \left( \int_0^\infty e^{-t} \, dX_t \right),
\end{equation}
(cf. \cite{jurekvervaat,wolfe}).
The Laplace exponents of $V$ and $X$ are related by
\begin{equation} \label{eq-laplaceonlydrift}
\psi_X(u)= u\psi_V'(u), \quad u> 0
\end{equation}
(e.g. \cite[Remark 4.3]{BarndorffNielsen-Shephard}; alternatively, \eqref{eq-laplaceonlydrift}
can be deduced from \eqref{eq-range-cond-0}).
Denoting the drifts of $V$ and $X$ by $b_V$ and $b_X$, respectively, it is easy to see that
\begin{equation} \label{eq-drift-relation}
b_V = b_X \int_0^\infty e^{-t} dt = b_X.
\end{equation}
Since the negative of the Laplace exponent of any infinitely divisible positive distribution is a Bernstein function and these are concave (cf. \cite[Definition 3.1 and Theorem 3.2]{rene-book}) it holds
$u\psi'(u)\geq \psi(u)$ for any such Laplace exponent. Together with the above we observe that
$\psi_X(u) \geq \psi_V(u)$ and hence $$\int_{(0,\infty)} (1-e^{-ut}) \, \nu_X(dt) \leq  \int_{(0,\infty)} (1-e^{-ut}) \, \nu_V(dt),\quad \forall\; u \geq 0.$$
Finally, the L\'evy density $x^{-1}k(x)$ of $V$ with $k$ non-increasing and the L\'evy measure
$\nu_X$ are related by
\begin{equation} \label{eq-k-X}
k(x) = \nu_X((x,\infty)), \quad x>0
\end{equation} (e.g. \cite[Equation (4.17)]{BarndorffNielsen-Shephard}).
In particular, the condition $k(0+) < \infty$ is equivalent to $\nu_X(\bR_+) < \infty$,
and the derivative of $-k$ is the L\'evy density of $\nu_X$.

\subsubsection*{Differential equation, necessary conditions, and nested ranges}

In the next result we give the differential equation for the Laplace transform of $V$,
which has to be satisfied if $\cL(V)$ is in the range $D_\xi^+$.
In the special case when $\eta$ is a compound Poisson process
with non-negative jumps, this differential equation \eqref{eq-diff-NilsenPaulsen} below has
already been obtained by Nilsen and Paulsen \cite[Proposition 2]{NilsenPaulsen1996}.
We then rewrite this differential equation in terms of $\psi_X$,
which turns out to be very useful for the further investigations.

\begin{theorem} \label{thm-range-description}
Let $\xi_t=\sigma B_t+a t$, $t\geq 0$, $\sigma, a >0$ for
some standard Brownian motion $(B_t)_{t\geq 0}$. Let $\mu =\law(V) \in L(\bR_+)$ have drift $b_V$
and L\'evy density given by $x^{-1} k(x)$, $x>0$, where $k:(0,\infty)\to [0,\infty)$
is non-increasing.
Then the following are true:
\begin{enumerate}
\item $\mu\in R_\xi^+$ if and only if there is some subordinator $\eta$ such that
\begin{equation} \label{eq-diff-NilsenPaulsen}
\frac12 \sigma^2 u^2 \LL_V''(u) + \left( \frac{\sigma^2}{2} - a \right) u \LL_V'(u)
+ \psi_\eta(u) \LL_V (u) = 0, \quad u > 0,
\end{equation}
in which case $\mu = \law(V) = \Phi_\xi(\law(\eta_1))$.
In particular, if $\eta$ is a subordinator,
then the Laplace transform
of $V$ satisfies \eqref{eq-diff-NilsenPaulsen} with $\LL_V(0) = 1$, and
if $V$ is not constant $0$, then $\lim_{u\to \infty} \LL_V(u) = 0$.
\item Let the subordinator $X=X(\mu)$ be related to $\mu$ by \eqref{eq-X-mu}.
Then $\mu\in R_\xi^+$ if and only if the function
$$
(0,\infty) \to \bR, \quad u \mapsto a \psi_X(u) - \frac{\sigma^2}{2} u \psi_X'(u)
- \frac{\sigma^2}{2} (\psi_X(u))^2
$$
defines the Laplace exponent $\psi_\eta(u)$ of some subordinator $\eta$.
In that case
\begin{equation} \label{eq-februar4}
\Phi_\xi (\law (\eta_1)) = \cL \left( \int_0^\infty e^{-t} \, dX_t \right) = \mu.
\end{equation}
\end{enumerate}
\end{theorem}

\begin{proof}
(i) By Theorem \ref{thm-range-condition}, $\mu=\cL(V) \in R_\xi^+$ if and only if
\begin{equation} \label{eq-laplacewithoutjumps}
\psi_\eta (u) = \left( a- \frac{\sigma^2}{2} \right) u \psi_V'(u) - \frac{\sigma^2}{2}
u^2 \left( \psi_V''(u) + (\psi_V'(u))^2 \right), \quad u > 0,
\end{equation}
for some subordinator $\eta$, in which case $\mu=\Phi_\xi(\law(\eta_1))$. Using \eqref{eq-derivative-1} and \eqref{eq-derivative-2},
it is easy to see that this is equivalent to \eqref{eq-diff-NilsenPaulsen}.
That $\LL_V(0)=1$ is clear. If $V$ is not constant 0, then it cannot have an atom at 0
(e.g. \cite[Theorem 2.2]{bertoinlindnermaller08}), hence $\lim_{u\to\infty} \LL_V(u) = 0$. \\
(ii)
If $\law(V) = \cL \left( \int_0^\infty e^{-t}\, dX_t\right) \in L(\bR_+)$
for some subordinator $X$, then by \eqref{eq-laplaceonlydrift}
$\psi_V'(u)= u^{-1}\psi_X(u)$ and $\psi_V''(u)= u^{-1}\psi'_X(u)- u^{-2}\psi_X(u)$.
Inserting this into \eqref{eq-laplacewithoutjumps} yields the condition
\begin{equation} \label{eq-rel-xitoX}
\psi_\eta(u) = a \psi_X(u) - \frac{\sigma^2}{2} u \psi_X'(u) - \frac{\sigma^2}{2} (\psi_X(u))^2,
\quad u > 0,
\end{equation}
which gives the claim.\qed
\end{proof}

\begin{remark} \label{rem-Riccati}
{\rm (i) Since $u \psi_X'(u) \geq \psi_X(u)$ as observed after Equation \eqref{eq-drift-relation}, it follows from \eqref{eq-rel-xitoX} that
$$\psi_\eta(u) \leq \left( a - \frac{\sigma^2}{2} \right) \psi_X(u) - \frac{\sigma^2}{2} (\psi_X(u))^2, \quad u>0,$$
when the subordinators $X$ and $\eta$ are related by \eqref{eq-februar4}.\\
(ii) Equation \eqref{eq-rel-xitoX} is  a Riccati equation for $\psi_X$.
Using the transformation \linebreak $y(u) = \exp (\int_1^u \frac{\psi_X(v)}{v} \, dv) = C \, \LL_V(u)$ for $u>0$
by \eqref{eq-laplaceonlydrift},
it is easy to see that it reduces to the linear equation \eqref{eq-diff-NilsenPaulsen}.
Unfortunately, in general it does not seem possible to solve \eqref{eq-diff-NilsenPaulsen} in a closed form.\\
(iii) Since for any subordinator $\eta$, $\psi_\eta(u)$ has a continuous continuation to
$\{z \in \CC : \Re(z) \geq 0\}$ which is analytic in $\{ z \in \mathbb{C}: \Re(z) > 0\}$
(e.g. \cite[Proposition 3.6]{rene-book}),
for any fixed $u_0>0$ Equation \eqref{eq-diff-NilsenPaulsen} can be solved in principle on
$(0,2u_0)$ by the power series method (e.g. \cite[Section 2.8, Theorem 7, p. 190]{braun1993}).
In particular when $\nu_\eta$ is such that
$\int_{(1,\infty)} e^{ux} \nu_\eta(dx) <\infty$ for every $u>0$ (e.g. if $\nu_\eta$ has compact support),
then $\psi_\eta(z) = - b_\eta z + \int_{(0,\infty)} (e^{- zx} -1) \, \nu_\eta(dx)$, $z\in \mathbb{C}$,
is an analytic continuation of $\psi_\eta$ in the complex plane.
Hence it admits a power series expansion of the form
$\psi_\eta(z) = \sum_{n=0}^\infty f_n z^n$, $z\in \mathbb{C}$,
with $f_0=0$ and Equation \eqref{eq-diff-NilsenPaulsen} may be solved by the Frobenius method (e.g. \cite[Sect. 2.8, Theorem 8, p. 215]{braun1993}).
To exemplify this, assume for simplicity that $2a/\sigma^2$ is not an integer.
Equation \eqref{eq-diff-NilsenPaulsen} has a weak singularity at $0$. Its so-called indicial polynomial
is given by
$$
r\mapsto r (r-1) + \left( 1 - \frac{2a}{\sigma^2} \right) r = r \left( r - \frac{2a}{\sigma^2} \right).
$$
The exponents of singularity are the zeros of this polynomial, i.e. $0$ and $2a/\sigma^2$, and
since we have assumed that $2a/\sigma^2$ is not an integer, the general real solution of
\eqref{eq-diff-NilsenPaulsen} is given by
$$
\LL_V(u) = C_1 u^{2a/\sigma^2} \sum_{n=0}^\infty c_n u^n + C_2 \sum_{n=0}^\infty d_n u^n, \quad u > 0,
$$
where $C_1,C_2\in \bR$, $c_0=d_0=1$, the coefficients $c_n,d_n$ are defined recursively by
$$
c_n := \frac{-1}{n (n+2a/\sigma^2)} \sum_{k=0}^{n-1} c_k f_{n-k}, \quad d_n = \frac{-1}
{n (n-2a/\sigma^2)} \sum_{k=0}^{n-1} d_k f_{n-k}, \quad n\in \mathbb{N},
$$
(e.g. \cite[Section 2.8, Equation (14), p. 209]{braun1993})
and the power series $\sum_{n=0}^\infty c_n u^n$ and $\sum_{n=0}^\infty d_n u^n$ converge in
$u\in \bC$.
Since $\LL_V(0) = 1$, we even conclude that $C_2=1$.
}
\end{remark}

Next, we show that the ranges of $\Phi_\xi$, when $\xi_t = \sigma B_t+at$,
are nested when $\sigma$ and $a$ vary over all positive parameters.

\begin{theorem} \label{thm-inclusion}
Let $B=(B_t)_{t\geq 0}$ be a standard Brownian motion. For $a,\sigma>0$
let $\xi^{(a,\sigma)}:= (\xi^{(a,\sigma)}_t)_{t\geq 0}:= (\sigma B_t + at)_{t\geq 0}$.
Then $R^+_{\xi^{(a,\sigma)}} = R^+_{\xi^{( a/\sqrt{\sigma},1)}}$.\\
Further, for $a,\sigma, a',\sigma'>0$ such that $a/\sqrt{\sigma} \leq a'/\sqrt{\sigma'}$
we have $R_{\xi^{(a,\sigma)}}^+ \subset R_{\xi^{(a',\sigma')}}^+$.\\
In particular, for fixed $\sigma>0$, the family $R_{\xi^{(a,\sigma)}}^+$, $a>0$,
is nested and non-decreasing in $a$, and for fixed $a>0$ the family
$R_{\xi^{(a,\sigma)}}^+$, $\sigma>0$, is nested and non-increasing in $\sigma$.
\end{theorem}

\begin{proof}
Since $(\sigma B_t + at)_{t\geq 0}$ has the same distribution as
$(B_{t\sqrt{\sigma}} + at)_{t\geq 0}$, we obtain for a L\'evy
process $\eta=(\eta_t)_{t\geq 0}$ such that $\law(\eta_1) \in D_{\xi^{(a,\sigma)}}$
and $\eta$ is independent of $B$,
$$\int_0^\infty e^{-(\sigma B_t + at)} \, d\eta_t \stackrel{d}{=} \int_0^\infty e^{-(B_{t\sqrt{\sigma}}
+ at)} \, d\eta_t = \int_0^\infty e^{-(B_t + (a/\sqrt{\sigma})t)} \, d\eta_{t/\sqrt{\sigma}}.$$
Hence $\law(\eta_{1/\sqrt{\sigma}}) \in D_{\xi^{(a/\sqrt{\sigma},1)}}$
and $\Phi_{\xi^{(a,\sigma)}}(\law(\eta_1)) = \Phi_{\xi^{(a/\sqrt{\sigma},1)}}
(\law (\eta_{1/\sqrt{\sigma})}))$. In particular,
$R_{\xi^{(a,\sigma)}}^+$ $\subset R_{\xi^{( a/\sqrt{\sigma},1)}}^+$.
Similarly, $R_{\xi^{(a,\sigma)}}^+ \supset R_{\xi^{( a/\sqrt{\sigma},1)}}^+$ so that
$R_{\xi^{(a,\sigma)}}^+ = R_{\xi^{( a/\sqrt{\sigma},1)}}^+$.
For the second assertion,
it is hence sufficient to assume $\sigma=1$.
Now if $a<a'$ and $\mu \in R_{\xi^{(a,1)}}^+$, let the subordinator $X$ be related to
$\mu$ by \eqref{eq-X-mu}.
Then
$$
a \psi_X(u) - \frac{1}{2} u \psi_X'(u) -\frac{1}{2} (\psi_X(u))^2 = \psi_\eta(u) , \quad u > 0,
$$
by Theorem \ref{thm-range-description}~(ii), hence
$$
a' \psi_X(u) - \frac{1}{2} u \psi_X'(u) -\frac{1}{2} (\psi_X(u))^2
= \psi_\eta(u) + (a'-a) \psi_X(u), \quad u > 0,
$$
defines the Laplace exponent of a subordinator by \cite[Corollary 3.8 (i)]{rene-book}.
Hence $\mu \in R_{\xi^{(a',1)}}^+$ again by Theorem \ref{thm-range-description}~(ii).
The remaining assertions are clear.\qed
\end{proof}

\begin{remark}
{\rm Although $R_{\xi^{(1,\sigma)}}^+ \subset R_{\xi^{(1,\sigma')}}^+$ for $0< \sigma' < \sigma$,
and $\sigma B_t +t$ converges pointwise to $t$ when $\sigma\to 0$,
we do not have $\bigcup_{\sigma > 0} R_{\xi^{(1,\sigma)}}^+ = R_{\xi_t=t}^+ \left( = L(\bR_+)\right)$.
For example, a positive $3/4$-stable distribution is in $L(\bR_+)$ but not
in $\bigcup_{\sigma > 0} R_{\xi^{(1,\sigma)}}^+$, as follows from Example \ref{ex-posstable} or Corollary
\ref{cor-almostnopositivestable} below.
}
\end{remark}

While it is difficult to solve the equations \eqref{eq-diff-NilsenPaulsen} and
\eqref{eq-rel-xitoX} for given $\psi_\eta$, they still allow to obtain results about
the qualitative structure of the range.
The following gives a simple necessary condition in terms of the L\'evy density $x^{-1}k(x)$ for $\mu$
to be in $R_\xi^+$, and to calculate the drift $b_\eta$ of $(\Phi_\xi^+)^{-1} (\mu)$
when $\mu\in R_\xi^+$.

\begin{theorem} \label{thm-growth}
Let $\xi_t = \sigma  B_t + at$, $t\geq 0$, for $\sigma, a>0$ and some standard Brownian motion
$B=(B_t)_{t\geq 0}$. Let $\mu=\law(V)\in L(\bR_+)$ with drift $b_V$ and L\'evy density $x^{-1}k(x)$.
Let the subordinator $X$ be related to $\mu$ by \eqref{eq-X-mu} and denote its drift by $b_X$.
\begin{enumerate}
\item If $\mu\in R_\xi^+$, then $b_X = 0$ and $\lim_{u\to \infty} u^{-1/2}
|\psi_X(u)|=\lim_{u\to\infty} u^{1/2} |\psi_V'(u)|$ exists and is finite.
If $\mu = \Phi_\xi (\law (\eta_1))$ for some subordinator $\eta$ with drift $b_\eta$,
then $b_\eta$ and $\psi_X$ are related by
\begin{equation} \label{eq-eta-drift}
b_\eta = \frac{\sigma^2}{2} \lim_{u\to\infty} u^{-1} (\psi_X(u))^2 = \frac{\sigma^2}{2}
\lim_{u\to\infty} u (\psi_V'(u))^2.
\end{equation}
\item If $\mu\in R_\xi^+$ has L\'evy density $x^{-1}k(x)$, then  it holds $\limsup_{x\downarrow 0} x^{-1/2} \int_0^x k(s)
\, ds < \infty$ and $b_V=0$. In particular, if  $\mu = \Phi_\xi (\law (\eta_1))$ for some subordinator $\eta$ with drift
$b_\eta$, then $b_\eta > 0$ if and only if $\limsup_{x\downarrow 0} x^{-1/2} \int_0^x k(s) \, ds > 0$.
\end{enumerate}
\end{theorem}

\begin{proof}
(i) Suppose that $\mu = \law(V) = \Phi_\xi(\cL(\eta_1))\in R_\xi^+$.
Then $b_V = 0$ by Lemma~\ref{lemma-no-drift} and hence $b_X=0$ by \eqref{eq-drift-relation}.
Since $\psi_X'(u) = - \int_{(0,\infty)} e^{-ux} x \, \nu_X(dx)$ we conclude that
$\lim_{u\to\infty} \psi_X'(u) = 0$ by dominated convergence.
Since $b_X=0$ and $\lim_{u\to\infty} u^{-1} \psi_X(u) = -b_X=0$ and
$\lim_{u\to\infty} u^{-1} \psi_\eta(u) = -b_\eta$ by \cite[Remark 3.3 (iv)]{rene-book},
\eqref{eq-eta-drift} as well as the necessity of the stated condition follow from
\eqref{eq-rel-xitoX}
and \eqref{eq-laplaceonlydrift}.\\
(ii) Since $k(x) = \nu_X((x,\infty))$ by \eqref{eq-k-X}, it follows from
\cite[Lemma 3.4]{rene-book} that
$$\frac{e-1}{e} \leq \frac{|\psi_X(u)|}{u \int_0^{1/u} k(s) \, ds } \leq 1, \quad u > 0.$$
Hence (ii) is an immediate consequence of (i) and Lemma \ref{lemma-no-drift}.\qed
\end{proof}

\begin{example}\label{ex-posstable}
{\rm Let $\xi_t =  \sigma B_t+at$ be as in Theorem \ref{thm-growth}. Let $\mu\in L(\bR_+)$ with L\'evy density $x^{-1}k(x)$.
Then $\int_0^1 k(x)\, dx < \infty$. \\ If $\liminf_{s\downarrow 0} k(s) s^{1/2} = +\infty$,
then $\liminf_{x\downarrow 0} x^{-1/2} \int_0^x k(s) \, ds = +\infty$.
Hence $\mu\not\in R_\xi^+$.
In particular, a non-degenerate positive $\alpha$-stable distribution with $\alpha>1/2$ cannot be
in $R_\xi^+$.  A more detailed result will be given in Corollary
\ref{cor-almostnopositivestable} below.}
\end{example}

\subsubsection*{Selfdecomposable distributions with $k(0+) < \infty$}

In this subsection we specialize to selfdecomposable distributions with $k(0+) < \infty$ and
give a characterization when they are in the range $R_\xi^+$ for $\xi$ a Brownian motion with drift.

\begin{theorem} \label{thm-range-finite}
Let $\xi_t = \sigma B_t + a t$, $t\geq 0$, $\sigma, a > 0$ for some standard Brownian
motion $(B_t)_{t\geq 0}$.
Let $\mu=\law(V)\in L(\bR_+)$ have drift $b_V$ and L\'evy density
$x^{-1} k(x)$, $x>0$, where $k=k_V:(0,\infty) \to [0,\infty)$ is non-increasing.
Let the subordinator $X=X(\mu)$ be related to $\mu$ by \eqref{eq-X-mu}.
Assume that $k(0+) < \infty$, equivalently that $\nu_X(\bR_+) < \infty$.
\begin{enumerate}
\item  Then $\mu\in R_\xi^+$ if and only if $b_X=0$ and $\nu_X$ has a density $g$ on $(0,\infty)$
such that
\begin{equation} \label{eq-cond-g1}
\lim_{t\to\infty} t g(t) = \lim_{t\to 0} t g(t) =0
\end{equation}
and such that the function
\begin{eqnarray} \label{eq-cond-g2}
G:(0,\infty) &\to & [0,\infty),\\
t&\mapsto& (a + \sigma^2 \nu_X(\bR_+)) \int_0^t g(v) \, dv
+ \frac{\sigma^2}{2} t g(t) - \frac{\sigma^2}{2} \int_0^t (g\ast g)(v) \, dv \nonumber
\end{eqnarray}
is non-decreasing. If these conditions are satisfied, then
$$
\Phi_\xi (\cL (\eta_1)) 
= \mu,
$$
where $\eta$ is the subordinator with drift $0$ and finite L\'evy measure $\nu_\eta(dx) = dG(x)$.
\item Equivalently, $\mu=\law(V)\in R_\xi^+$ if and only if $b_V=0$ and
$-k:(0,\infty) \to (-\infty,0]$ is absolutely continuous with derivative $g$ on
$(0,\infty)$ satisfying \eqref{eq-cond-g1} and such that $G$ defined by
\eqref{eq-cond-g2} is non-decreasing. In that case, $\Phi_\xi(\law(\eta_1))= \mu$, where $\eta$ is a subordinator with drift 0 and finite L\'evy measure $\nu_\eta(dx) = dG(x)$.
\end{enumerate}
\end{theorem}

\begin{proof}
(i) Assume that $\nu_X(\bR_+) < \infty$. Suppose first that $\mu\in R_\xi^+$, and let
$(\eta_t)_{t\geq 0}$ be a subordinator such that
$\Phi_\xi(\cL(\eta_1)) = \mu$.
Then $b_X=0$ by Theorem~\ref{thm-growth}~(i),
and by Theorem \ref{thm-range-description}~(ii), we have \eqref{eq-rel-xitoX}
with
\begin{eqnarray*}
\psi_\eta(u) &=& - b_\eta u - \int_{(0,\infty)} (1-e^{-ut}) \, \nu_\eta(dt) \\ 
\mbox{and} \quad
\psi_X(u) &=& - \int_{(0,\infty)} (1-e^{-ut}) \, \nu_X(dt), \quad u \geq 0.
\end{eqnarray*}
Since $\mathbb{L}_{\nu_X}(u)^2 = \mathbb{L}_{\nu_X \ast \nu_X} (u)$ and
$(\nu_X \ast \nu_X)(\bR_+) = \nu_X(\bR_+)^2$, where $\mathbb{L}_{\nu_X}$ denotes the Laplace transform of the finite measure $\nu_X$, we conclude
\begin{align*}
\psi_X(u)^2 & =  \left( \int_{(0,\infty)} (1-e^{-ut}) \, \nu_X(dt) \right)^2 \\
& =\nu_X(\bR_+)^2 - 2\nu_X(\bR_+) \int_{(0,\infty)} e^{-ut} \, \nu_X(dt) + \int_{(0,\infty)} e^{-ut}
(\nu_X \ast \nu_X) (dt) \\
& = \int_{(0,\infty)} (1-e^{-ut}) \, (2\nu_X(\bR_+) \nu_X - \nu_X \ast \nu_X)(dt).
\end{align*}
Hence, from \eqref{eq-rel-xitoX}, on the one hand
\begin{equation} \label{eq-rel-X-rho}
\frac{\sigma^2}{2} u \psi_X'(u)  = b_\eta u + \int_{(0,\infty)} (1-e^{-ut}) \,
\rho_1(dt) - \int_{(0,\infty)} (1-e^{-ut}) \, \rho_2(dt),\end{equation}
where
\begin{equation*}
\rho_1 := \nu_\eta + \frac{\sigma^2}{2} \nu_X \ast \nu_X \quad \mbox{and} \quad \rho_2
:= (a+\sigma^2 \nu_X(\bR_+)) \nu_X.
\end{equation*}
On the other hand, $u \psi_X'(u) = - u \int_{(0,\infty)} e^{-ut} t \, \nu_X(dt)$,
and rewriting the integral $\int_{(0,\infty)} (1-e^{-ut})\rho_i(dt) = \int_0^\infty u e^{-ut} \rho_i((t,\infty)) \, dt$
by Fubini's theorem as in \cite[Remark 3.3(ii)]{rene-book},
\eqref{eq-rel-X-rho} gives
$$
\frac{\sigma^2}{2} u \int_{(0,\infty)} e^{-ut} t \nu_X(dt) = -b_\eta u + u \int_0^\infty e^{- u t}
\left( \rho_2((t,\infty)) - \rho_1((t,\infty)) \right) \, dt, \quad u > 0.
$$
Dividing by $u$, the uniqueness theorem for Laplace transforms then shows $b_\eta=0$ and that $\nu_X$ has a density $g$, given by
\begin{equation} \label{eq-februar5}
g(t) = \frac{2}{\sigma^2 t} \left( \rho_2((t,\infty)) - \rho_1((t,\infty))\right), \quad t > 0.
\end{equation}
From this we conclude that $\lim_{t\to\infty} t g(t) = 0$ and that the limit $\lim_{t\to 0} t g(t) =
\frac{2}{\sigma^2} (\rho_2(\bR_+) - \rho_1(\bR_+))$ exists in $[-\infty,\infty)$ since
$\rho_2(\bR_+) < \infty$.
But since $g\geq 0$, the limit must be in $[0,\infty)$,
hence $\rho_1(\bR_+) < \infty$ so that $\nu_\eta(\bR_+) < \infty$, and since
$\int_0^1 \frac{t g(t)}{t} \, dt = \int_0^1 g(t) \, dt < \infty$, we also have
$\lim_{t\to 0} t g(t) = 0$. Further, by \eqref{eq-februar5}, the total variation of $t\mapsto t g(t)$ over $(0,\infty)$ is finite. Knowing now that $\nu_X$ has a density $g$ with $\lim_{t \to \infty}
t g(t) = \lim_{t\to 0} t g(t) = 0$, we can write using partial integration
\begin{align*}
u \psi_X'(u) & =  \int_0^\infty \left( \frac{d}{dt} e^{-ut} \right) t g(t) \, dt
=   \int_0^\infty  t g(t) d\left(e^{-ut}\right ) \\
&=  \left. t g(t) e^{-ut} \right|_{t=0}^{t=\infty} - \int_0^\infty e^{-ut} \, d(tg(t))
=  \int_0^\infty (1-e^{-ut}) \, d(tg(t)).
\end{align*}
Inserting this in \eqref{eq-rel-X-rho}, we obtain by uniqueness of the representation of
Bernstein functions (cf. \cite[Theorem 3.2]{rene-book}) that
$$
\frac{\sigma^2}{2} d(t g(t)) = \nu_\eta(dt) + \frac{\sigma^2}{2} (g\ast g) (t)\, dt
- (a + \sigma^2 \nu_X(\bR_+)) g(t) \, dt,
$$
or equivalently
\begin{equation} \label{eq-zwischen1}
\nu_\eta(dt) = (a + \sigma^2 \nu_X(\bR_+)) g(t) \, dt + \frac{\sigma^2}{2} d (t g(t)) -
\frac{\sigma^2}{2} (g\ast g) (t)\, dt.
\end{equation}
Since $\nu_\eta$ is a positive (and finite) measure, so is the right-hand side of
\eqref{eq-zwischen1}, and hence $G$ is non-decreasing with $\nu_\eta(dt) = d G(t)$,
finishing the proof of the \lq\lq only if\rq\rq-assertion.
The converse follows by reversing
the calculations above, by defining a subordinator $\eta$ with drift 0 and
L\'evy measure $\nu_\eta (dt) := dG(t)$, observing that $t\mapsto t g(t)$ is of finite total variation on $(0,\infty)$ by \eqref{eq-cond-g1} and \eqref{eq-cond-g2}, and then showing that $\nu_\eta$ satisfies
\eqref{eq-rel-X-rho} and hence that $\psi_\eta$ satisfies \eqref{eq-rel-xitoX}.\\
(ii) This follows immediately from (i), \eqref{eq-drift-relation} and \eqref{eq-k-X}.\qed
\end{proof}

\begin{remark} \label{rem-range-description}
{\rm Let $\xi_t = \sigma B_t + a t$, $t\geq 0$, with $\sigma, a>0$ and $(B_t)_{t\geq 0}$ a
standard Brownian motion.\\
(i) If $\mu \in R_\xi^+$ and $X$ is a subordinator such that \eqref{eq-X-mu} holds and such
that $\nu_X(\bR_+) < \infty$, then the L\'evy density $g$ of $\nu_X$ cannot have negative jumps,
since by \eqref{eq-cond-g2} this would contradict non-decreasingness of $G$.\\
(ii) Let $X$ be a subordinator with $\nu_X(\bR_+) < \infty$ and $b_X=0$, and suppose that $\nu_X$
has a density $g$ such that there is $r\geq 0$
with $g(t) = 0$ for $t\in (0,r]$ and $g$ is differentiable on $(r,\infty)$ (the case $r=0$
is allowed). Then $\cL(\int_0^\infty e^{-t} \, dX_t) \in R_\xi^+$ if and only if
$g$ satisfies \eqref{eq-cond-g1} and
\begin{equation} \label{eq-cond-g3}
\left( a+ \sigma^2 \nu_X(\bR_+) + \frac{\sigma^2}{2}\right) g(t) + \frac{\sigma^2}{2} t g'(t)
- \frac{\sigma^2}{2} (g\ast g) (t) \geq 0, \quad \forall\; t > r.
\end{equation}
This follows immediately from Theorem \ref{thm-range-description}~(iii) since the right-hand side
of \eqref{eq-cond-g3} is the derivative of the function $G$ defined by \eqref{eq-cond-g2}.}
\end{remark}

The following gives an example for a distribution in $R_\xi^+$ such that $\nu_X(\bR_+) < \infty$.
\begin{example}
{\rm Let $r\geq 0$ and let $g:[0,\infty) \to [0,\infty)$ be a function such that $g(t)=0$ for all
$t\in (0,r)$ (a void assumption if $r=0$), $g_{|[r,\infty)}$ is continuously differentiable with
derivative $g'$, such that $g$ is strictly positive on $[r,\infty)$, $\lim_{t\to\infty} g(t) = 0$
and such that $-g'$ is regularly varying at $\infty$ with index $\beta < -2$ (in particular,
$g'(t) < 0$ for large enough $t$).
Then $g$ defines a L\'evy density of a subordinator $X$ with drift 0 such that
$\nu_X(\bR_+) < \infty$ and
$\cL ( \int_0^\infty e^{-t} \, dX_t) \in R_{\sigma B_t + at}^+$ for large enough $a$.}
\end{example}

\begin{proof}
Since $-g'$ is regularly varying with index $\beta$ and $\lim_{t\to\infty} g(t) = 0$,
$g$ is regularly varying at $\infty$ with index $\beta + 1 < -1$ and $\lim_{t\to\infty}
\frac{-t g'(t)}{g(t)} = -\beta-1$ by Karamata's Theorem
(e.g. \cite[Theorem 1.5.11]{BinghamGoldieTeugels}).
In particular, $\lim_{t\to\infty} t g(t) = 0$,
further $\lim_{t\to 0} t g(t) = 0$ since $g(0) < \infty$, and $g$ is a density of a finite measure.
Next, observe that
$$
\frac{(g\ast g)(t)}{g(t)} = \int_r^{t/2} \frac{g(t-x)}{g(t)} g(x) \, dx
+ \int_{t/2}^{t-r} \frac{g(x)}{g(t)} g(t-x) \, dx, \quad t\geq 2r.
$$
But for any $\varepsilon>0$, when $t\geq t_\varepsilon$ is large enough,
we have $g(t-x)/g(t) \leq 2^{-\beta-1}+\varepsilon$ for $x\in (r,t/2]$, and
$g(x)/g(t)\leq 2^{-\beta-1} + \varepsilon$ for $x\in [t/2,t-r]$ by the uniform convergence theorem
for regularly varying functions (e.g. \cite[Theorem 1.5.2]{BinghamGoldieTeugels}).
As $\int_0^\infty g(t) \, dt <\infty$, this shows that $\limsup_{t\to\infty}
\frac{(g\ast g) (t)}{g(t)} < \infty$.
Since also $g\ast g$ as well as $|g'|$
are bounded on $[r,\infty)$, it follows that \eqref{eq-cond-g3} is satisfied for all $t\geq r$
for large enough $a$, and for $t\in (0,r)$ it is trivially satisfied.
Hence $\law(\int_0^\infty e^{-t} \, dX_t) \in R_{\sigma B_t+at}^+$ for large enough $a$.\qed
\end{proof}

Next we give some examples of selfdecomposable distributions which are not in $R_\xi^+$.

\begin{example}
{\rm  Let $\xi_t = \sigma B_t + at$, $t\geq 0$, with a standard Brownian motion
$B$ and parameters $\sigma,a > 0$.\\
(i) A selfdecomposable distribution with L\'evy density $c \mathds{1}_{(0,1)}(x) x^{-1}$
and $c>0$ is not in $R_\xi^+$ by Theorem
\ref{thm-range-finite}, since
$k(x) = \mathds{1}_{(0,1)}(x)$ satisfies $k(0+) < \infty$ but is not continuous.\\
(ii) If $X$ is a subordinator with non-trivial L\'evy measure $\nu_X$ such that
$\nu_X$ has compact support, then $\cL(\int_0^\infty e^{-t} \, dX_t)$
is not in $R_\xi^+$ by Theorem ~\ref{thm-range-finite}, since if it were then $\nu_X$ had a density $g$, and if $x_g$ denotes the right end
point of the support of $g$, then
$2x_g$ is the right endpoint of the support of $g\ast g$, showing that the function $G$
defined by \eqref{eq-cond-g2} cannot be non-decreasing on $(0,\infty)$.\\
(iii) If $X$ is a subordinator with finite L\'evy measure and non-trivial L\'evy density $g$ which is a step function
(with finitely or infinitely many steps), then $\law(\int_0^\infty e^{-t} \, dX_t)$
is not in $R_\xi^+$ by Remark \ref{rem-range-description}~(i), since $g$ must have at least one
negative jump as a consequence of $\int_0^\infty g(t) \, dt < \infty$.
}
\end{example}

\subsubsection*{Positive stable distributions}

In this subsection we characterize when a positive stable distribution is in the range
$R_\xi^+$. We also consider (finite) convolutions of positive stable distributions, i.e. distributions of the form $\law(\sum_{k=1}^n X_i)$, where $n\in \bN$
and $X_1,\ldots, X_n$ are independent positive stable distributions.

\begin{theorem} \label{thm-almostnopositivestable}
Set $\xi_t=\sigma B_t+a t$, $t\geq 0$, $a, \sigma >0$ for some standard
Brownian motion $(B_t)_{t\geq 0}$. Let $0<\alpha_1<\ldots<\alpha_n<1$ for some $n\in\NN$ and
$b_i\geq 0$, $i=1,\ldots,n$ and let $\mu$ be the distribution of $\sum_{i=1}^n X_i$ where the
$X_i$ are independent and each $X_i$ is non-trivial and positive $\alpha_i$-stable with drift $b_i$.
Then if $\mu$ is in $R_\xi^+$ it holds $b_i=0$, $i=0,\ldots, n$, $\alpha_1 \leq
(\frac{2a}{\sigma^2}\wedge \frac{1}{2})$ and $\alpha_n\leq \frac{1}{2}$.
Conversely, if $b_i=0$, $i=0,\ldots, n$ and $\alpha_n \leq (\frac{2a}{\sigma^2}\wedge \frac{1}{2})$,
then $\mu$ is in $R_\xi^+$.
\end{theorem}

\begin{proof}
Assume $\mu = \law(V)=\law(\int_0^\infty e^{-\xi_{s-}}d\eta_s)\in R_\xi^+$ for some subordinator $\eta$.
Since $\psi_V(u)=\sum_{i=1}^n \psi_{X_i}(u)$, the drift of $V$ is $\sum_{i=1}^n b_i$.
By Lemma~\ref{lemma-no-drift}, this implies $\sum_{i=1}^n b_i=0$ and hence $b_i=0$ for all $i$.
Since each $X_i$ is positive $\alpha_i$-stable with drift 0 and non-trivial, we know from
\cite[Remarks 14.4 and 21.6]{sato} that the Laplace exponent of $X_i$ is
given by
$$\psi_{X_i}(u)= \int_{(0,\infty)} (e^{-ux}-1)\nu_{X_i}(dx)=\int_0^\infty
(e^{-ux}-1) c_i x^{-1-\alpha_i} dx$$
with $c_i> 0$.
Hence
$$\psi_V(u)= \sum_{i=1}^n \int_0^\infty (e^{-ux}-1) c_i x^{-1-\alpha_i} dx,$$
such that 
\begin{align*}
\psi_V'(u)&
= -\sum_{i=1}^n c_i u^{\alpha_i-1} \Gamma(1-\alpha_i) \quad \mbox{and}\quad
\psi_V''(u)
= \sum_{i=1}^n c_i u^{\alpha_i-2} \Gamma(2-\alpha_i), \quad u > 0.
\end{align*}
Hence \eqref{eq-laplacewithoutjumps} reads
\begin{align}
\psi_\eta (u)
&= -\sum_{i=1}^n \left[ \left(\left(a -\frac{\sigma^2}{2}\right)c_i\,\Gamma(1-\alpha_i)+
\frac{\sigma^2}{2} c_i \,\Gamma(2-\alpha_i) \right) u^{\alpha_i} \right. \nonumber \\
&  \left.  \hskip 10mm  + \sigma^2 \sum_{j=1}^{i-1} c_i c_j \Gamma(1-\alpha_i) \Gamma(1-\alpha_j)
u^{\alpha_i+\alpha_j} +\frac{\sigma^2}{2} c^2_i (\Gamma(1-\alpha_i))^2 u^{2\alpha_i}  \right]
\nonumber \\
&=: - \sum_{i=1}^n \left(A_i u^{\alpha_i} + \sum_{j=1}^{i-1}B_{i,j}u^{\alpha_i+\alpha_j}
+ C_i u^{2\alpha_i}\right)=:-f(u), \quad u > 0.
\label{eq-laplace-aux1}
\end{align}
Observe that $A_i\in\RR$, and $B_{i,j}, C_i>0$ for all $i,j$.
As the left hand side of \eqref{eq-laplace-aux1} is the Laplace
exponent of a subordinator it is the negative of a Bernstein
function \cite[Theorem 3.2]{rene-book} and thus $f(u)$, $u\geq 0$,
has to be a Bernstein function if a solution to \eqref{eq-laplace-aux1} exists.
By \cite[Corollary 3.8 (viii)]{rene-book} a Bernstein function cannot grow faster than linearly,
which yields directly that $\alpha_i \in (0,1/2]$, $i=1,\ldots,n$.
As by \cite[Definition 3.1]{rene-book} the first derivative of a Bernstein function is completely
monotone, considering $\lim_{u\to 0} f'(u) \geq 0$ we further conclude that necessarily
$A_1\geq 0$, which is equivalent to $\alpha_1\leq \frac{2a}{\sigma^2}$.

Conversely, let $V$ be a non-trivial finite convolution of positive $\alpha_i$-stable distributions with
drift 0 and $0<\alpha_1< \ldots < \alpha_n\leq (\frac{2a}{\sigma^2} \wedge \frac12 )$.
Then $A_i\geq 0$ for all $i$ and the preceding calculations show that the right hand side of
\eqref{eq-laplacewithoutjumps} is given by $f(u)$, which is the Laplace exponent of a subordinator,
namely an independent sum of positive $\alpha_i$-stable subordinators (for each $A_i\geq 0$),
$(\alpha_i+\alpha_j)$-stable subordinators (for each $B_{i,j}$), $2\alpha_i$-stable subordinators
(for each $C_i$ with $\alpha_i<\frac12$) and possibly a deterministic subordinator (if $\alpha_n=1/2$).
Hence $\law (V) \in R_\xi^+$ by Theorem \ref{thm-range-condition}.\qed
\end{proof}

As a consequence of the above theorem, we can characterize which positive $\alpha$-stable distributions are in $R_\xi^+$:

\begin{corollary} \label{cor-almostnopositivestable}  Let $\xi_t=\sigma B_t+a t$, $t\geq 0$, $a, \sigma >0$ for some standard
Brownian motion $(B_t)_{t\geq 0}$. Then a non-degenerate positive $\alpha$-stable distribution $\mu$ is in $R_\xi^+$ if and only if its drift is $0$ and $\alpha \in (0,\frac{2a}{\sigma^2}\wedge \frac{1}{2}]$. If this condition is satisfied and $\mu$ has L\'evy density $x\mapsto c x^{-1-\alpha}$ on $(0,\infty)$ with $c>0$, then
$\mu=\Phi_\xi(\law(\eta_1))$, where in the case $\alpha<1/2$, $\eta$ is a subordinator with drift 0 and L\'evy density on $(0,\infty)$ given by
$$x\mapsto c\alpha \left(  a - \frac{\sigma^2}{2} \alpha \right)  x^{-\alpha-1} + \sigma^2 c^2 \frac{\alpha (\Gamma(1-\alpha))^2}{\Gamma(1-2\alpha)} x^{-2\alpha-1},$$
and in the case $\alpha=1/2 = 2a/\sigma^2$, $\eta$ is a deterministic subordinator with drift $\sigma^2 c^2 (\Gamma(1-\alpha))^2/2$.
\end{corollary}

\begin{proof}
The equivalence is immediate from Theorem \ref{thm-almostnopositivestable}. Further,
by \eqref{eq-laplace-aux1}, we have $\Phi_\xi(\law(\eta_1)) = \mu$ where the Laplace exponent of $\eta$ is given by
$$\psi_\eta (u) =
-\left(\left(a -\frac{\sigma^2}{2}\right)c\,\Gamma(1-\alpha)+
\frac{\sigma^2}{2} c \,\Gamma(2-\alpha) \right) u^{\alpha}
 -\frac{\sigma^2}{2} c^2 (\Gamma(1-\alpha))^2 u^{2\alpha} .$$
The case $\alpha=1/2=2a/\sigma^2$ now follows immediately, and for $\alpha<1/2$ observe that
 \begin{align*}\int_0^\infty (e^{-ux}-1) x^{-1-\beta} \, dx & = \int_0^u \left(\frac{d}{dv} \int_0^\infty
 (e^{-vx}-1) x^{-1-\beta} \, dx\right) \, dv \\
  & = -\int_0^u v^{\beta-1} \Gamma(1-\beta) \, dv =
 -\frac{\Gamma(1-\beta)}{\beta} u^\beta
 \end{align*}
for $\beta \in (0,1)$ und $u>0$, which gives the desired form of the drift and  L\'evy density of $\eta$ also in this case.\qed
\end{proof}

\begin{example} \label{rem-Dufresne2}
{\rm Reconsider Example \ref{rem-Dufresne}, namely,
$$
V=\int_0^\infty e^{-(\sigma B_t+a t)} dt\overset{d}= \frac{2}{\sigma^2 \Gamma_{\frac{2a}{\sigma^2}}},
$$
where $V$ has the law of a scaled inverse Gamma distributed random variable
with parameter $\frac{2a}{\sigma^2}$. In the case that $\frac{2a}{\sigma^2}=\frac{1}{2}$, or equivalently $a=\sigma^2/4$
this is a so called L\'evy distribution and it is $1/2$-stable (cf. \cite[p. 507]{steutelvanharn}).
Reassuringly, by Corollary \ref{cor-almostnopositivestable}, $\law(V)$ is a $1/2$-stable distribution if $a=\sigma^2/4$.
}
\end{example}

\begin{corollary}
Let $\xi_t=\sigma B_t+a t$, $t\geq 0$, $\sigma, a >0$ for some standard
Brownian motion $(B_t)_{t\geq 0}$.
Then $R_\xi^+$ contains the closure of all finite convolutions of positive $\alpha$-stable distributions
with drift $0$ and $\alpha\in (0,\frac{2a}{\sigma^2} \wedge \frac12 ]$,
which is characterized as the set of infinitely divisible distributions $\mu$ with Laplace exponent
\begin{equation} \label{eq-feb8-1}
\psi(u)= \int_{(0,\frac{2a}{\sigma^2} \wedge \frac12 ]} m(d\alpha) \int_0^\infty \left(e^{-ux} - 1
\right) x^{-1-\alpha} dx \, \end{equation}
where $m$ is a measure on $(0, \frac{2a}{\sigma^2} \wedge \frac12]$ such that
\begin{equation} \label{eq-feb8-2}\int_{(0,\frac{2a}{\sigma^2} \wedge \frac12 ]} \alpha^{-1} m(d\alpha) <\infty.
\end{equation}
\end{corollary}
\begin{proof}
Denote by $M_1$ the class of all finite convolutions of positive $\alpha$-stable
distributions with drift 0 and $\alpha\in (0,\frac{2a}{\sigma^2} \wedge \frac12 ]$, by $M_2$ its closure with respect to weak convergence, and by $M_3$ the class of all positive distributions on $\bR$ whose characteristic exponent can be represented in the form \eqref{eq-feb8-1} with $m$ subject to \eqref{eq-feb8-2}. We show that $M_2=M_3$, then since  $M_2\subset R_\xi^+$ by
Theorems \ref{thm-almostnopositivestable} and \ref{thm-cont-inverse1} (i), this implies the statement. To see $M_2\subset M_3$, denote by $L_\infty(\bR)$ the closure of
all finite convolutions of stable distributions on $\bR$ (cf. \cite[Theorem 3.5]{sato-classL}, where $L_\infty(\bR)$ is defined differently, but shown to be equivalent to this definition). Using the fact that $L_\infty(\bR)$ is closed, it then follows easily from \cite[Theorem 4.1]{sato-classL} that also $M_3$ is closed under weak convergence. Since obviously $M_1\subset M_3$ (take $m$ to be a measure supported on a finite set), we also have $M_2\subset M_3$. Conversely, $M_3\subset M_2$ can be shown in complete analogy to the proof of \cite[Theorem 3.5]{sato-classL}. \qed
\end{proof}

\begin{remark}
{\rm  From the proof of Theorem \ref{thm-almostnopositivestable} it is possible to obtain a necessary and sufficient condition for a finite convolution of positive, stable distributions to be in $R_\xi^+$. Indeed if the $X_i$ are such that $\psi_{X_i}(u)=-c_i u^{\alpha_i}$ with $c_i>0$ and $\alpha_i \in (0,1)$, then $\mu=\cL(\sum_{i=1}^n X_i)$ is in $R_\xi^+$ if and only if the function $f$ defined by \eqref{eq-laplace-aux1} is a Bernstein function.
 After ordering the indices, the function $f$ can be written as $\sum_{i=1}^m D_i u^{\gamma_i}$ with $0<\gamma_1< \ldots \gamma_m<2$ and coefficients $D_i \in \bR\setminus \{0\}$. Since
 $$ \sum_{i=1,\ldots, m;  \gamma_i < 1} D_i u^{\gamma_i} = \int_0^\infty (1-e^{-ux}) \sum_{i=1,\ldots, m; \gamma_i< 1} \frac{D_i \gamma_i}{\Gamma(1-\gamma_i)} x^{-1-\gamma_i} \, dx$$ as seen in the proof
 of Corollary \ref{cor-almostnopositivestable}, it follows from \cite[Corollary 3.8(viii)]{rene-book} and \cite[Example  12.3]{sato} that $f$ is a Bernstein function if and only if $\gamma_m\leq 1$, $D_m\geq 0$ and
 $$\sum_{i=1,\ldots, m; \gamma_i < 1} \frac{D_i \gamma_i}{\Gamma(1-\gamma_i)} x^{-1-\gamma_i}\geq 0, \quad \forall\; x > 0.$$}
 \end{remark}

\section*{Acknowledgement}
We would like to thank the anonymous referee for valuable suggestions which helped to improve the exposition of the manuscript. 
Makoto Maejima's research was partially supported by JSPS Grand-in-Aid for Science Research 22340021. 
%

\end{document}